\documentclass[hidelinks,onefignum,onetabnum]{siamart251216}

\usepackage{amsfonts}
\usepackage{amsmath} 
\usepackage{amssymb} 

\usepackage{graphicx}

\usepackage{booktabs}
\usepackage{multirow}
\usepackage{threeparttable}
\usepackage{pgfplots}
\usepackage{microtype}

\usepackage{comment}
\usepackage{tikz}
\usetikzlibrary{shapes, arrows.meta, positioning, shadows, calc}

\usepackage{algpseudocode}
\usepackage{algorithm}
\usepackage[bb=dsserif]{mathalpha}
\usepackage{bm}

\usepackage{xcolor,hyperref}
\definecolor{darkblue}{rgb}{0.0,0.0,0.6}
\hypersetup{colorlinks,breaklinks,linkcolor=darkblue,urlcolor=darkblue,anchorcolor=darkblue,citecolor=darkblue}

\ifpdf
\DeclareGraphicsExtensions{.eps,.pdf,.png,.jpg}
\else
\DeclareGraphicsExtensions{.eps}
\fi

\newsiamremark{remark}{Remark}
\newsiamremark{example}{Example}
\newsiamremark{problem}{Problem}
\newsiamremark{assumption}{Assumption}
\newsiamremark{hypothesis}{Hypothesis}
\crefname{hypothesis}{Hypothesis}{Hypotheses}
\newsiamthm{claim}{Claim}
\newsiamremark{fact}{Fact}
\crefname{fact}{Fact}{Facts}

\headers{Sampled-Data Distributionally Robust Control}{C.-H. Hsieh}

\title{Sampled-Data Wasserstein Distributionally Robust Control of Multiplicative Systems: A Convex Relaxation with Performance Guarantees\thanks{
	 {This paper was supported in part by the National Science and Technology Council (NSTC), Taiwan, under Grants: NSTC113--2628--E--007--015-- and NSTC114--2628--E-007--006--.} 
	}
}

\author{
Chung-Han Hsieh\thanks{Department of Quantitative Finance, National Tsing Hua University, Hsinchu, Taiwan (\email{ch.hsieh@mx.nthu.edu.tw}).}
}

\usepackage{amsopn}

\ifpdf
\hypersetup{
pdftitle={An Example Article},
pdfauthor={D. Doe, P. T. Frank, and J. E. Smith}
}
\fi

\begin{document}

\maketitle

\begin{abstract}
	This paper investigates the robust optimal control of sampled-data stochastic systems with multiplicative noise and distributional ambiguity.  We consider a class of discrete-time optimal control problems where the controller \emph{jointly} selects a feedback policy and a sampling period to maximize the worst-case expected  concave utility of the inter-sample growth factor. Modeling  uncertainty via a Wasserstein ambiguity set, we confront the structural obstacle of~``concave-max'' geometry arising from maximizing a concave utility against an adversarial distribution. Unlike standard convex loss minimization, the dual reformulation here requires a minimax interchange within the semi-infinite constraints where the utility's concavity precludes exact strong duality. To address this, we utilize a general minimax inequality to derive a tractable convex relaxation. Our approach yields a rigorous lower bound that functions as a probabilistic performance guarantee. We establish an explicit, non-asymptotic bound on the resulting duality gap, proving that the approximation error is uniformly controlled by the Lipschitz-smoothness of the stage reward and the diameter of the disturbance support. Furthermore, we introduce necessary and sufficient conditions for \emph{robust viability}, ensuring state positivity invariance across the entire ambiguity set.  Finally, we bridge the gap between static optimization and dynamic performance, proving that the optimal value of the relaxation serves as a rigorous deterministic floor for the asymptotic average utility rate almost surely.
	The framework is illustrated on a log-optimal portfolio control problem, which serves as a canonical instance of multiplicative stochastic control.
\end{abstract}



\section{Introduction} \label{section: Introduction}
	This paper investigates the robust optimal control of sampled-data stochastic systems subject to multiplicative noise and distributional ambiguity. Specifically, we consider a class of discrete-time optimal control problems in which the controller jointly selects a feedback policy and a sampling period to maximize the worst-case expected utility of the inter-sample growth factor. By integrating the Wasserstein metric into a sampled-data framework, we develop a distributionally robust control~(DRC) strategy that ensures \emph{system viability} (state positivity) and achieves robust performance despite epistemic uncertainty in the disturbance distribution.
	 Such stochastic systems with multiplicative noise arise naturally in growth-dominated environments and have been studied extensively; see, e.g., \cite{aoki1975control, gershon2001h, farina2000positive, hsieh2023asymptotic}.
	 Related implementation constraints---including sample-and-hold operation, communication losses, and data-driven controller synthesis---have also been widely investigated in networked and learning-based control; see, e.g., \cite{sinopoli2004kalman, hespanha2007survey, coppens2020data}.

	For concreteness, we use the \emph{Optimal Growth Portfolio} (Kelly) problem in quantitative finance as a running example; see \cite{Kelly_1956, Breiman_1961, cover2012elements, karatzas2021portfolio}. In this context, the system state represents wealth, the multiplicative noise represents asset returns, and the sampling period corresponds to the rebalancing frequency. Importantly, finance serves primarily as a concrete instantiation: the central theoretical challenges---specifically, the coupling between control frequency, transaction costs (actuation friction), and distributional robustness---are fundamental to a broad class of engineering and economic systems driven by multiplicative noise; see, e.g., \cite{hsieh2019positive} for related positivity considerations under implementation delays.

\subsection{Motivation and Literature Review} 
The classical theory of stochastic optimal control is by now well established; see \cite{pham2009continuous, bertsekas2012dynamic}. Numerous extensions and refinements have since been developed to address modeling complexity and uncertainty. In particular, stochastic systems with multiplicative (random-coefficient) effects have motivated a rich body of modeling and control paradigms; see, e.g., \cite{aoki1975control, gershon2001h}. In parallel, robust formulations based on minimax (worst-case) criteria have been proposed for stochastic control problems; see, e.g., \cite{hinrichsen1998stochastic, gonzalez2002minimax}. It is well known that, in the presence of proportional friction, continuous-time limits of such problems typically lead to singular stochastic control formulations \cite{davis1990portfolio}. By contrast, the sampled-data framework proposed here avoids these singularities while retaining rigorous and tractable performance bounds.

\medskip
\paragraph{Sampled-Data Constraints and Friction} 
Real-world controllers operate in discrete time via sample-and-hold mechanisms. A central feature of sampled-data control and stochastic rebalancing systems is that the choice of the sampling period $n$ induces a fundamental trade-off; see \cite{aastrom2013computer, korn2002stochastic}.
While classical results primarily address the preservation of stability under sampling, \cite{nesic2004framework}, we focus instead on growth and distributional~ambiguity.

Indeed, high-frequency control allows for closer tracking of an ideal continuous-time policy, but incurs high accumulated friction (actuation costs), which in our framework is captured through the state-dependent inter-sample growth factor $\Phi_n$. Conversely, low-frequency control reduces frictional costs but increases discretization error and exposure to open-loop drift. 
As a result, the sampling decision directly affects~$\Phi_n$ through both disturbance aggregation and accumulated friction, making the control and sampling period choices structurally non-separable in $\Phi_n$.
Foundational results on sampled-data systems with stochastic sampling periods can be traced back to \cite{de1988stationary}, while LMI-based approaches for multiplicative-noise control were developed in \cite{el1995state}. Recent applications have begun to treat the sampling period itself as a decision variable \cite{kuhn2010analysis,wong2023frequency, hsieh2023asymptotic}.
These methods build upon earlier theoretical frameworks for robust growth optimality under moment uncertainty \cite{rujeerapaiboon2016robust}, yet they typically assume that the driving disturbance distribution is perfectly known.

\medskip
\paragraph{Distributional Ambiguity} 
Distributionally robust optimization (DRO) provides a principled framework for hedging against model misspecification  by optimizing against an ambiguity set of probability measures. Such ambiguity sets can be constructed using a variety of metrics, including the Prohorov metric~\cite{erdougan2006ambiguous}, box-type convex polyhedral uncertainty~\cite{ben2009robust}, Kullback-Leibler divergence~\cite{ben2013robust, hu2013kullback}, and moment-based descriptions~\cite{delage2010distributionally}.  In this work, we focus on Wasserstein ambiguity sets, which are particularly well-suited for data-driven settings and admit finite-sample performance guarantees and asymptotic consistency under mild conditions~\cite{mohajerin2018data, blanchet2019quantifying, gao2023distributionally}.

In the multiperiod Markov decision process literature, Wasserstein-based distributionally robust control has been studied in \cite{yang2020wasserstein}, where a Bellman operator reformulation yields a contraction mapping and enables multi-stage out-of-sample guarantees.
Recent extensions have addressed tractability and partial observability in linear–quadratic settings via approximation techniques; see, e.g., \cite{hakobyan2024wasserstein}.

In parallel, Wasserstein DRO has also been employed in growth-optimal decision problems arising in finance, leading to robust ``Wasserstein–Kelly’’ rules; see, e.g., \cite{rujeerapaiboon2016robust, li2023wasserstein, hsieh2023solving}. However, these formulations typically abstract away from transaction costs and from the strategic choice of the rebalancing horizon---two features that are central under sampled-data implementation and directly couple control actions with the sampling decision. In contrast, our setting involves concave utility maximization for state-multiplicative systems, where the standard convex–concave structure required for minimax equality generally fails, and exact saddle-point reformulations are unavailable.

\medskip
\paragraph{The Concave-Max Geometry} 
Maximizing concave utilities is closely connected to risk-sensitive stochastic control; see, e.g., \cite{whittle1990risk, fleming1995risk}. In that literature, exponential (entropic) utilities often yield tractable dynamic programming recursions, whereas general concave utilities under distributional ambiguity---as considered here---require fundamentally different techniques.

A key theoretical gap arises when applying standard Wasserstein DRO machinery to utility maximization for state-multiplicative systems. Classical finite-sample Wasserstein DRO results, e.g., \cite{mohajerin2018data, gao2023distributionally, kuhn2025distributionally}, are predominantly developed for minimizing a convex loss function, a setting that guarantees minimax \emph{equality} via Sion's Minimax Theorem under the usual convexity–concavity assumptions. In contrast, in our setting the stage utility is concave in the disturbance, and tractability hinges on a minimax interchange embedded within the semi-infinite constraints of the dual representation. Because the required convex–concave conditions fail for this constraint-level interchange, an exact equality is generally unavailable. We therefore rely on the general minimax inequality to construct a conservative convex relaxation that yields a rigorous lower bound on the worst-case expected utility, together with an explicit non-asymptotic bound on the relaxation gap induced by this interchange.

\medskip
\subsection{Contributions} 
To the best of our knowledge, this is the first framework to unify frequency-aware sampled-data robust control with Wasserstein distributional robustness for multiplicative systems. Our specific contributions are as follows:

\emph{A Unified Sampled-Data DRO Formulation:} We formulate the joint optimization of the control input $u$ and sampling period $n$ as a worst-case expected utility maximization problem. We introduce the concept of \emph{Robust Viability} (Lemma \ref{lemma: robust viability condition}), a rigorous invariance condition ensuring that the system state remains strictly positive for all distributions within the Wasserstein ambiguity set, a prerequisite for the well-posedness of the utility maximization criterion. From a control-theoretic perspective, the robust viability constraint plays a role analogous to safety invariance in multiplicative systems subject to model uncertainty.

\emph{Tractable Relaxation via Minimax Inequalities with Probabilistic Guarantees:} Addressing the ``Concave-Max'' difficulty, we derive a computationally tractable convex relaxation of the infinite-dimensional control problem (Theorem \ref{theorem:convex reduction}). In contrast to standard DRO results that typically achieve exact reformulation via minimax equality, we rigorously frame our solution as a \emph{lower bound} on the true optimal value, utilizing a general minimax inequality.	Then, Lemma~\ref{lemma: statistical_bound} provides a probabilistic guarantee showing that this conservatism yields a valid lower confidence bound on the true performance. In safety-critical control, maximizing the floor is often preferred over maximizing the mean.

\emph{Long-Run Performance Guarantees:} Bridging the gap between static optimization and dynamic performance, we establish a theoretical link between the solution of the convex relaxation and the asymptotic behavior of the system (Theorem \ref{theorem:long_run_utility}). We prove that, under ergodic assumptions, the optimal value of our tractable formulation serves as a deterministic floor for the long-run average utility rate almost surely.
 For the specific case of log-utility, this result guarantees a certified floor for the asymptotic capital growth rate, providing a rigorous theoretical justification for the rolling-horizon implementation; see Corollary~\ref{corollary: growth_guarantee}. This result connects our finite-horizon robust optimization to the  theory of ergodic control; see, e.g., \cite{arapostathis1993discrete, borkar1988control}.

\emph{Explicit Non-Asymptotic Duality Gap Analysis:} 
We provide a theoretical certificate for the quality of our approximation. In Proposition \ref{proposition: Minimax Duality Gap Bound}, we derive an explicit, non-asymptotic upper bound on the minimax duality gap. We show that this gap scales with the Lipschitz smoothness of the stage reward function and the diameter of the disturbance support, but is independent of the ambiguity radius. This result justifies the use of the convex relaxation by quantifying the maximum potential suboptimality.

\emph{Empirical Validation and Performance Analysis:} We demonstrate the practical efficacy of the framework on a log-optimal portfolio control problem using a dataset of major large-cap S\&P 500 assets. By implementing the distributionally robust controller via a cutting-plane scheme, we show that jointly optimizing the sampling period and the feedback policy significantly outperforms standard market benchmarks in terms of downside risk and risk-adjusted returns.

\medskip
The remainder of this paper is structured as follows. Section \ref{section: Problem Formulation} details the model components. Section \ref{section: Main Results} presents our main theoretical results, including the tractable reformulation, the duality gap analysis, and various performance guarantees. Section~\ref{section: Illustrative Examples} presents an illustrative example in quantitative finance, and Section~\ref{section: conclusion} concludes the~paper.

\section{Preliminaries and Problem Formulation} \label{section: Problem Formulation}
This section formulates a class of sampled-data stochastic \emph{multiplicative} control systems under distributional ambiguity.

\subsection{Notation} \label{subsection: notations}
Throughout the paper, $\mathbb{R}$  denotes the set of real numbers, $\mathbb{R}_{+}$ the set of nonnegative real numbers, and $\overline{\mathbb{R}} := \mathbb{R} \cup \{-\infty, \infty\}$ the extended reals.
All random objects are defined on a probability space~$(\Omega, \mathcal{F}, \mathbb{P}).$ 
The notation~$\|\cdot \|$ denotes the $\ell_r$-norm on~$\mathbb{R}^k$ (where dimension $k$ is determined by context) for a fixed~$r \in [1, \infty]$ chosen throughout the paper. We write $\| \cdot \|_*$ for the dual norm associated with $\|\cdot\|$, defined as~$\| z \|_* := \sup_{x \in \mathbb{R}^k} \{  z^\top x  : \| x \| \leq 1 \} $ for~$z \in \mathbb{R}^k$.  
We denote by $\delta_x$ the Dirac measure concentrating unit mass at~$x.$ The product of two probability distributions~$\mathbb{P}_1$ and $\mathbb{P}_2$ on~$\mathfrak{X}_1$ and $\mathfrak{X}_2$, respectively, is the distribution $\mathbb{P}_1 \otimes \mathbb{P}_2$ on $\mathfrak{X}_1 \times \mathfrak{X}_2.$

\subsection{Sampled-Data Multiplicative Dynamics}\label{subsection:dynamics}
Consider a discrete-time stochastic control system with a fixed sampling period $n \geq 1$. Let~$k \in \{0, 1, 2, \dots\}$ denote the discrete time index, which corresponds to the physical time interval~$[kn, (k+1)n)$. We denote the scalar system state at the beginning of the~$k$th interval by $V_k := V(kn) \in \mathbb{R}_{+}$, and the control input by $u_k \in \mathcal{U} \subset \mathbb{R}^m$. The control is implemented via a \emph{zero-order hold} mechanism, remaining constant throughout the interval~$[kn, (k+1)n)$. The system evolves according to the \emph{multiplicative} state~dynamics:
\begin{align}\label{eq:mult_dynamics_general}
	V_{k+1} = V_k \, \Phi_n(u_k, \mathcal{X}_{k,n}), \qquad k=0,1,2,\dots,
\end{align}
where $\mathcal{X}_{k,n}$ represents the \emph{aggregated exogenous input (disturbance)} over the $k$th interval, taking values in a compact set $\mathfrak{X}_n \subset \mathbb{R}^d$. The function $\Phi_n: \mathcal{U} \times \mathfrak{X}_n \to (0, \infty)$ is the state-transition map, representing the inter-sample \emph{growth factor}. This function captures the coupled effects of the control decision and the uncertainty, implicitly accounting for actuation friction (e.g., efficiency losses, budget costs, etc.) and the multiplicative forcing of the external environment.

\begin{remark}[Dimensionality Mismatch] \rm
	The formulation explicitly allows for distinct control and disturbance dimensions ($m \neq d$). This flexibility captures diverse scenarios, such as \emph{under-actuation} ($m < d$), where a low-dimensional control (e.g., a scalar leverage ratio) faces high-dimensional noise, or \emph{redundancy} ($m > d$), such as a large portfolio allocation driven by a few low-dimensional latent factors.
\end{remark}

\begin{figure}[htbp]
	\centering
	\resizebox{0.7\textwidth}{!}{%
		\begin{tikzpicture}[
			font=\small,
			>=Stealth, thick,
			block/.style={draw, rectangle, align=center, minimum height=3.2em, minimum width=6.0em, fill=blue!3},
			smallblock/.style={draw, rectangle, align=center, minimum height=2.8em, minimum width=5.0em, fill=blue!3},
			dottedsep/.style={dotted, gray}
			]
			
			\node[block] (ctrl) {DRO Controller\\ $u_k = u^*$ \\ $n = n^*$};
			\node[smallblock, right=2.2cm of ctrl] (zoh) {ZOH\\ $u(t) \equiv u_k$};
			\node[block, above=1.7cm of ctrl] (data) {Data Buffer/History};
			
			\node[block, right=2.6cm of zoh] (plant) {Multiplicative update\\ $V_{k+1}=V_k\,\Phi_n(u_k,\mathcal{X}_{k,n})$};
			
			\node[smallblock, above=1.6cm of plant] (agg)
			{Aggregator\\ exogenous disturbances \\ on $[t_k, t_{k+1})$ $\mapsto\ \mathcal{X}_{k,n}$};
			
			\node[smallblock, below=2.0cm of zoh] (sampler) {Sampler\\ $t_{k+1} = t_k  + n$};
 
			\draw[dottedsep] ($(zoh.north east)!0.5!(plant.north west) + (0,3.0)$)
			-- ($(zoh.south east)!0.5!(plant.south west) + (0,-3.0)$);
			\node[gray, above=0.9cm of zoh] {\footnotesize \emph{Discrete-time decision}};
			\node[gray, above=0.9cm of plant] {\footnotesize \emph{Inter-sample environment}};
			
			\draw[->] (ctrl) -- node[above] {$u^*$} (zoh);
			\draw[->] (zoh) -- node[above] {$u_k$} (plant);
			\draw[->] (data) -- node[right] { $\widehat{\mathbb{F}}_n$} (ctrl);		
 

			\draw[->] ($(agg.north)+(0,0.9)$) -- node[right] {Disturbances on $[t_k, t_{k+1})$} (agg);
			\draw[->] (agg) -- node[right] {$\mathcal{X}_{k,n}$} (plant);
			\draw[->] (agg) -- node[above] {$\widehat{ \mathcal{X}}_{n}^{(j)}$} (data);
			
			\draw[->] (plant) -- ++(3.8,0) coordinate (out);
			\node[above] at ($(plant.east)+(1.0,0)$) {$V_{k+1}$};
			
			\draw[->] ($(plant.east)+(1.2,0)$) |- (sampler.east);
			\draw[->, dashed, gray] (sampler.west) -| node[pos=0.65, above] {\qquad \quad $V_{k+1}$} (ctrl.south);
			
		\end{tikzpicture}
	}
	\caption{Schematic of the sampled-data loop. The analysis uses the sampled state $V_k=V(t_k)$ and an aggregated disturbance $\mathcal{X}_{k, n}$ over $[t_k, t_{k+1})$ where $t_k = kn.$}
	\label{fig:control_loop}
\end{figure}
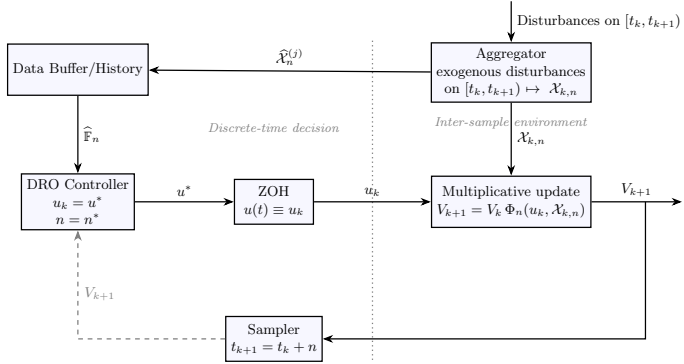

\subsection{Performance Criterion and State Positivity}\label{subsec:utility_viability}
Let $U :\mathbb{R}_{+} \to \mathbb{R}$ be a concave, non-decreasing utility function.
We measure performance over a sampling period by the \emph{stage reward} given by
	$
		r_n(u, x):=U( \Phi_n(u, x) ).
	$
A canonical example of a risk-averse growth factor corresponds to the logarithmic utility~$U(t)=\log t$.

To ensure the problem is well-posed despite the uncertainty, we enforce a viability constraint.
Fix a viability margin $\eta>0$ such that $[\eta, \infty)\subseteq \mathrm{dom}(U)$. We define the \emph{set of
	admissible controls} as:
\begin{equation}\label{eq:viable_set_generalU}
	\mathcal U_{\rm v}(n;\eta)
	:= \big\{ u \in \mathcal{U} : \Phi_n(u,x) \ge \eta, \quad \forall x \in \mathfrak{X}_n \big\}.
\end{equation}
This condition guarantees that the stage reward $r_n(u,x)$ is well-defined and finite uniformly over the disturbance support $\mathfrak{X}_n$.

\subsection{Standing Assumptions on Growth Factor}\label{subsec:assumptions}
To ensure computational tractability and derive an explicit non-asymptotic bound on the duality gap, we impose the following structural assumptions on the system data and the utility function.

\begin{assumption}[Regularity and Smoothness]\label{ass:U_phi}
	For each fixed sampling period $n \geq 1$, we assume the following conditions:
	\begin{enumerate}
		\item[(A1)] (\emph{Compactness}) The admissible control set $\mathcal{U} \subset \mathbb{R}^m$ is nonempty, compact, and convex. The disturbance support $\mathfrak{X}_n$ is nonempty and compact.
		
		\item[(A2)] (\emph{Continuity and Separate Concavity})
		The stage reward map 
		$
		(u,x) \mapsto r_n(u, x) 
		$
		is continuous and concave in $u$ for each $x \in\mathfrak{X}_n$, and concave in $x$ on $ \operatorname{conv}(\mathfrak{X}_n)$ for each $u \in \mathcal{U}_{\rm v}(n; \eta)$, where ${\rm conv}(\cdot)$ denotes the convex hull.
		
	\item[(A3)] (\emph{Uniform $L_n$-Smoothness})  
	For each $u \in \mathcal{U}_{\rm v}(n; \eta)$, the map $x \mapsto r_n(u, x)$ is differentiable on $\mathfrak{X}_n$ and its gradient $\nabla_x r_n(u, x)$ satisfies a uniform Lipschitz condition: there exists a constant $L_n > 0$ such that
	\[
		\sup_{u \in \mathcal{U}_{\rm v}(n; \eta)} \ 
		\sup_{\substack{x,y\in\mathfrak{X}_n\\ x\neq y}}
		\frac{\big\| \nabla_x r_n(u, x) - \nabla_x r_n(u, y) \big\|_*}{\| x - y\|}
		 \leq L_n .
	\]
	\end{enumerate}

\end{assumption}

\begin{remark}\rm
	Assumption~\ref{ass:U_phi}(A2) preserves the concave-max character of the problem. A sufficient condition for (A2) is that $\Phi_n(u,x)$ is concave in $(u,x)$ and $U$ is concave and nondecreasing on $[\eta, \infty)$.  Assumption~\ref{ass:U_phi}(A3) is the only assumption linking the duality gap bound to the specific choice of utility. The following example demonstrates that Assumption~\ref{ass:U_phi}(A3) holds for the affine-logarithmic case.
\end{remark}

\begin{example}[Log-Optimal Portfolio Control] \label{example: log-optimal portfolio control}
	Let $u  \in \mathcal{U}_{\rm v}(n;\eta)$ denote the vector of portfolio weights invested in risky assets for some $\eta > 0$. 
	Consider the affine growth factor (common in portfolio optimization, e.g., see \cite{hsieh2023solving}) $\Phi_n(u, x) = u^\top x + c_n(u)$, and the utility $U(y) = \log(y)$. The stage reward is
	$
	r_n(u, x) = \log( u^\top x + c_n(u) ).
	$
	The gradient with respect to the disturbance $x$ is
	$
	\nabla_x r_n(u, x) = \frac{u}{u^\top x + c_n(u)} = \frac{u}{\Phi_n(u, x)}.
	$
	Note that since $u\in\mathcal U_{\rm v}(n;\eta)$, we have $\underline{\Phi}_n(u):=\min_{x\in\mathfrak X_n}\Phi_n(u,x)\ge\eta$.
	Thus, for any~$x,y \in \mathfrak{X}_n$,
	\begin{align*}
		\|\nabla_x r_n(u,x)-\nabla_x r_n(u,y)\|_*
		&=\left\|u\left(\frac{1}{\Phi_n(u,x)}-\frac{1}{\Phi_n(u,y)}\right)\right\|_* \\
		&\leq \|u\|_* \,\frac{|\Phi_n(u,y)-\Phi_n(u,x)|}{\Phi_n(u,x)\,\Phi_n(u,y)} \\
		&=\|u\|_* \,\frac{|u^\top(y-x)|}{\Phi_n(u,x)\,\Phi_n(u,y)} \\
		&\leq \|u\|_*\,\frac{\|u\|_*\,\|y-x\|}{\underline{\Phi}_n(u)^2}
	= \frac{\|u\|_*^2}{\underline{\Phi}_n(u)^2}\,\|x-y\|,
	\end{align*}
	where the last inequality holds by the generalized Cauchy-Schwarz inequality $|u^\top z| \leq \|u\|_* \|z\|$, e.g., see~\cite{beck2017first}.
	Hence $\nabla_x r_n(u, x)$ is $L_n(u)$-Lipschitz on $\mathfrak X_n$ with
	$
	L_n(u)=\frac{\|u\|_*^2}{\underline{\Phi}_n(u)^2}.
	$
	In particular, since $\underline{\Phi}_n(u) \geq \eta$, we have $L_n(u)\le \frac{\|u\|_*^2}{\eta^2}$, and thus
	\[
	\sup_{u\in\mathcal U_{\rm v}(n;\eta)} L_n(u) \leq \sup_{u\in\mathcal U_{\rm v}(n;\eta)}\frac{\|u\|_*^2}{\eta^2}.
	\]
	Since $\mathcal U_{\rm v}(n;\eta)\subseteq \mathcal U$ and $\mathcal U$ is compact by (A1), $\sup_{u\in\mathcal U_{\rm v}(n;\eta)}\|u\|_*<\infty$; therefore a finite uniform constant $L_n:= \sup_{u\in\mathcal U_{\rm v}(n;\eta)} L_n(u)$ exists, verifying Assumption~\ref{ass:U_phi}(A3) for the affine-log case.
\end{example}

\subsection{Wasserstein Ambiguity Set}
As discussed in Section~\ref{section: Introduction}, the true disturbance distribution is often unknown. To hedge against this uncertainty without the excessive conservatism of worst-case robust optimization, we model distributional uncertainty using the Wasserstein metric. 
Let~$\mathcal{M}(  \mathfrak{X}_n )$  be the space of all probability distributions~$\mathbb{F}$ supported on~$ \mathfrak{X}_n$.
Since $\mathfrak{X}_n$ is compact, all such distributions have finite moments.

\begin{definition}[Wasserstein Metric] \label{definition:Wasserstein metric} \rm
	For $p \in [1, \infty)$ and any two distributions $\mathbb{F}^1, \mathbb{F}^2 \in \mathcal{M}( \mathfrak{X}_n)$, the \emph{p-Wasserstein metric}~$W_p: \mathcal{M}( \mathfrak{X}_n) \times \mathcal{M}( \mathfrak{X}_n) \to \mathbb{R}$ induced by the ground norm $\|\cdot\|$ fixed in Section~\ref{subsection: notations} is defined~by
	\begin{align*}
		W_p(\mathbb{F}^1, \mathbb{F}^2)
		& := \left( \inf_{\Pi \in \Gamma(\mathbb{F}^1, \mathbb{F}^2)}     \mathbb{E}_{( \mathcal{X}^1, \mathcal{X}^2) \sim \Pi} [\|\mathcal{X}^1 - \mathcal{X}^2\|^p ]  \right)^{1/p}
	\end{align*}
	where $\Gamma(\mathbb{F}^1, \mathbb{F}^2)$ is the set of joint distributions (couplings) on $\mathfrak{X}_n \times \mathfrak{X}_n$ with marginals $\mathbb{F}^1$ and $\mathbb{F}^2$.  
\end{definition}

Throughout the paper, we fix an order $p \in [1,\infty)$ and use the $p$-Wasserstein metric~$W_p$. All Wasserstein ambiguity sets in this paper are defined with respect to the same ground norm $\|\cdot\|$.
Consistent with data-driven formulations, e.g., see~\cite{calafiore2013direct, mohajerin2018data, yang2020wasserstein},
let $\widehat{\mathcal{X}}_n^{(1)}, \dots, \widehat{\mathcal{X}}_n^{(N_n)} \in \mathfrak{X}_n$ denote the observed samples of the $n$-period aggregated disturbance, where $N_n$ is the number of samples available at horizon $n$. The associated empirical distribution is
$
\widehat{\mathbb{F}}_n := \frac{1}{N_n}\sum_{j=1}^{N_n}\delta_{\widehat{\mathcal{X}}_n^{(j)}},
$
where $\delta_x$ denotes the Dirac measure at~$x$.

For a radius $\varepsilon \geq 0$, the \emph{Wasserstein ambiguity set} is defined as the ball of radius $\varepsilon$ centered at $\widehat{\mathbb{F}}_n$:
\begin{align} \label{def: Wasserstein ball}
	\mathcal{B}_\varepsilon^{(p)}( \widehat{\mathbb{F}}_n ) := \left\{ \mathbb{F} \in \mathcal{M}( { \mathfrak{X} }_n ) : W_p (\mathbb{F}, \widehat{\mathbb{F}}_n) \leq \varepsilon \right\}.
\end{align}

To ensure statistical validity, the ambiguity radii are calibrated specifically for each sampling period as follows.

\begin{definition}[Calibrated Ambiguity Radii] \label{def: calibrated_radius} \rm
	Fix a global confidence level $\beta \in (0,1)$ and a finite candidate set $\mathcal N \subseteq \mathbb{N}$.
	For each $n \in \mathcal{N}$, choose a radius $\varepsilon_n \geq 0$ such that
	\begin{align}\label{eq:coverage_union}
		\mathbb{P} \left(\mathbb F_{\mathrm{true}, n} \in \mathcal{B}_{\varepsilon_n}^{(p)} (\widehat{\mathbb F}_n)\right)  \geq 1-\frac{\beta}{|\mathcal N|},
	\end{align}
	where $|\mathcal{N}|$ denotes the cardinality of the set $\mathcal{N}$, and $\mathbb{F}_{\mathrm{true}, n}$ denotes the true distribution of the generic $n$-period aggregated disturbance $\mathcal{X}$.\footnote{ 
		Such radii can be obtained from concentration bounds, e.g., see \cite{fournier2015rate}, under suitable sampling assumptions, or calibrated empirically (e.g., via block bootstrap in the presence of serial dependence).
		}
	Consequently, by a union bound,
	\[
	\mathbb{P} \left( \bigcap_{n \in \mathcal{N}} \{ \mathbb F_{\mathrm{true}, n}\in \mathcal{B}_{\varepsilon_n}^{(p)} (\widehat{\mathbb F}_n) \} \right) \geq 1-\beta.
	\]
\end{definition}

\begin{remark} \rm
		Note that for the generic ball definition, if $\varepsilon = 0$, then $\mathcal{B}_\varepsilon^{(p)}( \widehat{\mathbb{F}}_n )  = \{ \widehat{\mathbb{F}}_n \}$, a singleton empirical distribution. 
\end{remark}

\subsection{Robust Viability and Admissible Controls} \label{subsection: Robust Viability}
In the context of multiplicative systems, state positivity is a prerequisite for the well-posedness of the performance criterion. Similar to stability or invariance requirements in classical control, see \cite{chen1998linear, farina2000positive}, we require the control $u$ to guarantee that the state trajectory remains strictly positive uniformly over the ambiguity set. 

We formalize this via the concept of \emph{robust viability}. Recall the set of admissible controls $\mathcal{U}_{\rm v}(n;\eta)$ defined in \eqref{eq:viable_set_generalU}. The following lemma establishes that this static constraint on the growth factor is sufficient to guarantee dynamic state invariance under distributional uncertainty.

\begin{lemma}[Robust Viability Condition] \label{lemma: robust viability condition} \it
	Fix a sampling period $n \ge 1$. Let $V_0 > 0$. If a control $u$ is admissible, i.e., $u \in \mathcal{U}_{\rm v}(n;\eta)$, then for any ambiguity radius $\varepsilon \ge 0$ and any distribution $\mathbb{F} \in \mathcal{B}_\varepsilon^{(p)}(\widehat{\mathbb{F}}_n)$, the state evolution satisfies 
	\[
	\mathbb{P}^\mathbb{F}( V_{k+1} \ge \eta V_k ) = 1 \quad \text{and} \quad \mathbb{P}^\mathbb{F}( V_{k+1} > 0 ) = 1, \quad \forall k \ge 0.
	\]
\end{lemma}
\begin{proof}
	Fix $\varepsilon \ge 0$ and $\mathbb{F} \in \mathcal{B}_\varepsilon^{(p)}(\widehat{\mathbb{F}}_n)$. By definition, the ambiguity set is a subset of $\mathcal{M}(\mathfrak{X}_n)$, meaning every feasible distribution is supported on the compact set $\mathfrak{X}_n$.
	Condition~$u \in \mathcal{U}_{\rm v}(n;\eta)$ implies that $\Phi_n(u, \mathcal{X}) \ge \eta$ holds almost surely with respect to any such $\mathbb{F}$, where $\mathcal{X}$ denotes a generic $n$-period aggregated disturbance.
	Since the dynamics are multiplicative, i.e., $V_{k+1} = V_k \Phi_n(u, \mathcal{X})$, if $V_k > 0$, it follows almost surely that $V_{k+1} \ge \eta V_k$. Given $V_0 > 0$ and $\eta > 0$, strictly positive invariance follows by induction.
\end{proof}

While robust viability is a hard constraint, one might consider a probabilistic relaxation. We define this notion to clarify the hierarchy of constraints.

\begin{definition}[$(\varepsilon, \delta)$-Viability] \rm \label{definition: epsilon_delta viability}
	Fix $n \ge 1$, $\varepsilon \geq 0$, and $\delta \in (0,1)$. A control $u \in \mathcal{U}$ is \emph{$(\varepsilon, \delta)$-viable} if
	$
	\inf_{\mathbb{F} \in \mathcal{B}_\varepsilon^{(p)}(\widehat{\mathbb{F}}_n)} \mathbb{P}^{\mathbb{F}} \big( \Phi_n(u, \mathcal{X}) \ge \eta \big) \ge 1-\delta.
	$
\end{definition}

 \begin{remark}\rm
 	Unlike robust viability, $(\varepsilon,\delta)$-viability does not impose $\Phi_n(u,x)\ge \eta$ uniformly over $x\in\mathfrak X_n$.
 	Consequently, for utilities with a singularity at zero (e.g., $U=\log$), $(\varepsilon,\delta)$-viability alone does not in general ensure that
 	$\mathbb E^{\mathbb F}[U(\Phi_n(u,\mathcal X))]$ is finite for all $\mathbb F\in\mathcal B_\varepsilon^{(p)}(\widehat{\mathbb F}_n)$.
 \end{remark}

The following theorem establishes the hierarchy between the hard robust constraint and the probabilistic chance constraint.

\begin{theorem}[Hierarchy of Viability Conditions]\label{theorem: Implication Hierarchy}
	The following implications hold for a given $n \ge 1$:
	
	$(i)$ If $u \in \mathcal{U}_{\rm v}(n;\eta)$, then it is $(\varepsilon, \delta)$-viable for \emph{all} $\varepsilon \ge 0$ and $\delta \in (0,1)$.
	
	$(ii)$ Conversely, if $u$ is $(\varepsilon, \delta)$-viable for every $\varepsilon \geq 0$ and every $\delta \in (0,1)$, then
	\[
	\inf_{x \in \mathfrak{X}_n} \Phi_n(u, x) \geq \eta.
	\]
\end{theorem}
\begin{proof}  
	$(i)$ If $\inf_{x \in \mathfrak{X}_n} \Phi_n(u, x) \ge \eta$, then the event $\{\Phi_n(u, \mathcal{X}) \geq \eta\}$ occurs almost surely under any distribution supported on $\mathfrak{X}_n$. Thus, its probability is 1 under any distribution supported on $\mathfrak{X}_n$, satisfying the condition for any $\delta$.
	
	$(ii)$ We proceed by contradiction. Suppose $u$ is $(\varepsilon, \delta)$-viable for all $\varepsilon, \delta$, but there exists $x^* \in \mathfrak{X}_n$ such that $\Phi_n(u, x^*) < \eta$. Let $\mathbb{F}^* = \delta_{x^*}$ be the Dirac measure at~$x^*$. 
	Since $x^* \in \mathfrak{X}_n$, $\mathbb{F}^* \in \mathcal{M}(\mathfrak{X}_n)$.
	Consider the distance $W_p(\mathbb{F}^*, \widehat{\mathbb{F}}_n)$. Since $\mathfrak{X}_n$ is compact, this distance is finite. Choose an ambiguity radius $\varepsilon' \geq W_p(\mathbb{F}^*, \widehat{\mathbb{F}}_n)$. Then $\mathbb{F}^*  \in \mathcal{B}_{\varepsilon'}^{(p)}(\widehat{\mathbb{F}}_n)$.
	Under such $\mathbb{F}^*$, the event $\{\Phi_n(u, \mathcal{X}) \geq \eta\}$ has probability zero. This implies
	$$
	\inf_{\mathbb{F} \in \mathcal{B}_{\varepsilon'}^{(p)}(\widehat{\mathbb{F}}_n)} \mathbb{P}^{\mathbb{F}} \big( \Phi_n(u, \mathcal{X}) \ge \eta \big) = 0,
	$$
	which violates the condition that $u$ is $(\varepsilon', \delta)$-viable for any $\delta < 1$.
\end{proof}

Based on this hierarchy, throughout the remainder of this work we impose the robust viability constraint
$u \in \mathcal{U}_{\rm v}(n; \eta)$ as defined in~\eqref{eq:viable_set_generalU}.
Since $\mathfrak{X}_n$ is compact and~$x \mapsto \Phi_n(u,x)$ is continuous for each fixed $u$, this condition is equivalently written~as
\begin{align} \label{ineq: robustly viable condition min form}
	\min_{x\in\mathfrak X_n}\Phi_n(u,x) \geq \eta.
\end{align}
In particular, the minimum is attained. Moreover, if $u\mapsto \Phi_n(u,x)$ is concave for each $x$, then
$g(u) := \min_{x \in \mathfrak{X}_n}\Phi_n(u,x)$ is concave as a pointwise infimum of concave functions; hence its superlevel set $\{u \in \mathcal{U}: \ g(u) \geq \eta\}=\mathcal U_{\rm v}(n;\eta)$ is convex.

\subsection{Distributionally Robust Control Formulation}
We are now ready to state the sampled-data distributionally robust control problem. The objective is to maximize the worst-case expected utility of the growth factor over the sampling period.

\begin{problem}[Horizon-Consistent Distributionally Robust Control] \rm \label{problem: general DRO problem}
	Fix a global confidence level $\beta\in(0,1)$ and a viability margin $\eta > 0$.
	Let $\mathcal N\subset\mathbb N$ be a finite candidate set of sampling periods, and for each $n \in \mathcal{N}$
	construct $\widehat{\mathbb F}_n$ from $N_n$ samples and choose $\varepsilon_n$ according to Definition~\ref{def: calibrated_radius}.
	The joint optimization over the sampling period $n$ and control $u$ is
	\begin{align}\label{eq: general_DRO_formulation}
		\max_{n\in\mathcal N}\ \max_{u\in \mathcal U_{\rm v}(n; \eta) }\
		\inf_{\mathbb F\in\mathcal{B}_{\varepsilon_n}^{(p)}(\widehat{\mathbb F}_n)}\ \frac{1}{n}\,
		\mathbb E^{\mathbb F}\!\left[U(\Phi_n(u,\mathcal X))\right].
	\end{align}
\end{problem}

Problem~\eqref{eq: general_DRO_formulation} is a mixed-integer optimization problem due to the discrete variable~$n$. However, for a fixed sampling period~$n$, the problem of maximizing
$$
	\inf_{\mathbb{F} \in \mathcal{B}_{\varepsilon_n}^{(p)}( \widehat{\mathbb{F}}_n )}\, \mathbb{E}^{\mathbb{F}} \left[ U(\Phi_n(u, \mathcal{X})) \right]
$$
over $u \in \mathcal{U}_{\rm v}(n; \eta)$ is a concave maximization problem, provided $U ( \Phi_n(u, \cdot))$ is concave. The global solution is obtained by solving these finitely many tractable subproblems and selecting the optimal sampling period $n^*$.

\begin{remark}[Relation to Chance Constraints] \rm
	We enforce the robust viability constraint $u \in  \mathcal{U}_{\rm v}(n;\eta)$ rather than the probabilistic $(\varepsilon, \delta)$-viability constraint. This is a deliberate design choice for tractability: $ \mathcal{U}_{\rm v}(n;\eta)$ is a deterministic convex set, preserving the convexity of the overall problem. In contrast, distributionally robust chance constraints typically induce non-convex feasible sets, which would render the problem computationally intractable. In addition, for logarithmic and other utilities with singular behavior at zero, enforcing $\Phi_n(u,x)\ge\eta>0 $ uniformly over $ \mathfrak{X}_n$ ensures that the expected utility is finite under all distributions in the ambiguity set, thereby guaranteeing well-posedness of the objective. Relaxing to a distributional chance constraint would require restricting to utilities bounded below, which excludes the classical log utility case.
\end{remark}

\begin{remark}[Reduction to Sample Average Approximation] \rm
	If the ambiguity radius is set to $\varepsilon = 0$, the set $\mathcal{B}_0^{(p)}(\widehat{\mathbb{F}}_n)$ collapses to the singleton empirical distribution~$\widehat{\mathbb{F}}_n$. Problem~\eqref{eq: general_DRO_formulation} then reduces to the Sample Average Approximation (SAA) of the risk-sensitive control problem:
	\[
		\max_{n \in \mathcal N}\ \max_{u \in  \mathcal U_{\rm v}(n; \eta) }\ \frac{1}{n}\frac{1}{N_n}\sum_{j=1}^{N_n} U\!\left(\Phi_n\!\left(u, \widehat{\mathcal X}_n^{(j)}\right)\right).
	\]
	This recovers standard expected-utility maximization but retains the hard viability constraint essential for validity.
\end{remark}

\section{Theoretical Results} \label{section: Main Results}
This section demonstrates that the infinite-dimensional distributionally robust control problem~\ref{problem: general DRO problem} can be approximated by a finite  dimensional convex program with a concave objective, thereby facilitating computational tractability.

\subsection{Tractable Formulation}
Although our dual derivation shares ingredients with~\cite{mohajerin2018data}, the polarity of optimization is reversed: we solve a max-min risk-sensitive problem where the objective is concave in the decision variable. Consequently, Sion’s minimax theorem cannot be directly applied to establish strong duality. We instead utilize a general minimax inequality to derive a rigorous lower bound.

\begin{theorem}[Tractable Convex Relaxation for Fixed Sampling Period $n$] \label{theorem:convex reduction} 
	Fix a sampling period $n \in \mathcal{N}$ and let $\varepsilon_n \ge 0$ be the calibrated ambiguity radius defined in Definition~\ref{def: calibrated_radius}. Let $p \in [1, \infty)$ and $q= \frac{p}{p-1}$ with $q=\infty$ if $p=1$. A tractable lower bound is given by the optimal value of the following convex program:
	\begin{align} \label{problem:DRO-ELG dual}
		J_{\rm  cvx}^*(n) := &  \sup_{u, \lambda, s_j, z_j} \frac{1}{n} \left(- \lambda\varepsilon_n^p + \frac{1}{N_n} \sum_{j=1}^{N_n} {s}_j \right)   \\
		\text{\rm s.t. }  &
		{ \min_{ x \in \mathfrak{X}_n} \left[ U(\Phi_n(u, x)) + z_j^\top ( x -\widehat{\mathcal{X}}_n^{(j)} ) \right] - \Omega_p(z_j, \lambda) \geq {s}_j, \quad \forall j=1, \dots, N_n, } \notag \\
		& \lambda \geq 0, \notag\\
		& u \in  \mathcal{U}_{\rm v}(n;\eta). \notag
	\end{align}
	where the regularization term $\Omega_p(z_j, \lambda)$ is defined as:
	$$
	\Omega_p(z_j, \lambda) := 
	\begin{cases} 
		\frac{1}{q}(p\lambda)^{1-q} \|z_j\|_*^q & \text{if } p > 1, \\ 0 & \text{if } p=1, \, \|z_j\|_* \le \lambda, \\ \infty & \text{otherwise},
	\end{cases}
	$$
	where $z_j\in\mathbb R^d$ are dual variables and $\|\cdot\|_*$ is the dual norm associated with the original norm $\|\cdot\|$ used in the Wasserstein metric. 	For $p>1$, we adopt the extended-value convention at $\lambda=0$: $\Omega_p(0,0)=0$ and $\Omega_p(z,0)=+\infty$ for $z\neq 0$.
\end{theorem}
\begin{proof}
	Using the Wasserstein distance from Definition~\ref{definition:Wasserstein metric}, we first re-express the inner minimization problem of Problem~\ref{problem: general DRO problem} via Lagrangian dualization. Specifically, we observe that
	\begin{align} \label{problem:the inner minimization problem}
		\inf_{\mathbb{F} \in \mathcal{B}_{\varepsilon_n}^{(p)} ( \widehat{\mathbb{F}}_n)} \mathbb{E}^{\mathbb{F}} \left[ U(\Phi_n(u, \mathcal{X})) \right]
		& = \begin{cases}
			\displaystyle\inf_{\Pi, \mathbb{F}} \displaystyle\int U(\Phi_n(u, \mathcal{X})) \, d\mathbb{F}(\mathcal{X})\\
			\text{s.t. }
				 \Pi \in \Gamma(\mathbb{F},\widehat{\mathbb{F}}_n), \\
				\qquad \int \| \mathcal{X} - \mathcal{X}^{\prime} \|^p \, d\Pi(\mathcal{X}, \mathcal{X}^\prime) \leq \varepsilon_n^p,
		\end{cases}
	\end{align}
	where $\Pi$ is a joint distribution on $\mathfrak{X}_n \times \mathfrak{X}_n$ with marginals $\mathbb{F}$ and $\widehat{\mathbb{F}}_n$.
	Given that the empirical distribution $\widehat{\mathbb{F}}_n$ is discrete, by the standard disintegration theorem, e.g., see \cite{mohajerin2018data}, any coupling $\Pi$ admits a decomposition into conditional distributions~$\mathbb{F}_j$ supported on $\mathfrak{X}_n$, associated with each sample $\widehat{\mathcal{X}}_n^{(j)}$. Consequently, the infinite-dimensional optimization problem \eqref{problem:the inner minimization problem}  over probability measures reduces to a finite sum of subproblems:
	\begingroup
	\allowdisplaybreaks
	{\small \begin{align}
			&\begin{cases}
				\displaystyle \inf_{\mathbb{F}_j \in \mathcal{M}(\mathfrak{X}_n), \, \forall j} \, \frac{1}{N_n}\displaystyle\sum_{j=1}^{N_n} \int U(\Phi_n(u, \mathcal{X})) \, d\mathbb{F}_j(\mathcal{X}) \\
				\text{s.t. } \displaystyle \frac{1}{N_n}\sum_{j=1}^{N_n} \int \| \mathcal{X} - \widehat{\mathcal{X}}_n^{(j)}\|^p \, d\mathbb{F}_j(\mathcal{X}) \leq \varepsilon_n^p
			\end{cases} \notag \\
			&  = \sup_{\lambda \geq 0}\; \inf_{\mathbb{F}_j \in \mathcal{M}(\mathfrak{X}_n), \, \forall j} \left\{-\lambda \varepsilon_n^p + \frac{1}{N_n} \sum_{j=1}^{N_n} \int \left[ U(\Phi_n(u, \mathcal{X})) + \lambda \| \mathcal{X} - \widehat{\mathcal{X}}_n^{(j)}\|^p \right] d{\mathbb{F}_j} \right\}  \label{eq:strong_duality_step}\\
			&  = \sup_{\lambda \geq 0} \left\{ -\lambda \varepsilon_n^p + \frac{1}{N_n} \sum_{j=1}^{N_n} \inf_{ x \in \mathfrak{X}_n} \left[ U(\Phi_n(u, x)) + \lambda \| x - \widehat{\mathcal{X}}_n^{(j)}\|^p \right] \right\}, \label{eq:reduction_to_pointwise}
		\end{align}
	}\endgroup
	where Equality~\eqref{eq:strong_duality_step}  holds as follows: If $\varepsilon_n>0$, the primal problem admits a strictly feasible point $\widehat{\mathbb F}_n$ with
	$W_p(\widehat{\mathbb F}_n,\widehat{\mathbb F}_n) = 0 < \varepsilon_n$, hence strong duality follows; see \cite{mohajerin2018data, gao2023distributionally}.
	If $\varepsilon_n=0$, then $\mathcal{B}_{\varepsilon_n}^{(p)}(\widehat{\mathbb{F}}_n)=\{\widehat{\mathbb F}_n\}$ and the reduction follows directly. 

	Introducing epigraphical auxiliary variables ${s}_j$ for~$j = 1, \dots, N_n$, we rewrite the problem as:
	\begin{align}
		&\begin{cases}
			\displaystyle\sup_{\lambda, {s}_j} \,  -\lambda \varepsilon_n^p + \frac{1}{N_n}  \sum_{j=1}^{N_n} {s}_j \\
			\text{s.t. } \displaystyle\inf_{ x \in \mathfrak{X}_n} \left[ U(\Phi_n(u, x)) + \lambda  \| x - \widehat{\mathcal{X}}_n^{(j)}\|^p   \right] \geq {s}_j, \quad \forall j \\
			\qquad \lambda \geq 0 \notag
		\end{cases} \\
		& = \begin{cases}
			\label{eq: double dual norm representation}
			\displaystyle\sup_{\lambda, {s}_j} \, -\lambda\varepsilon_n^p + \frac{1}{N_n}  \sum_{j=1}^{N_n} {s}_j \\
			\text{s.t. } { \displaystyle \inf_{ x \in \mathfrak{X}_n} \left[ U(\Phi_n(u, x)) + \sup_{z_j} \left\{ z_j^\top \left( x - \widehat{\mathcal{X}}_n^{(j)} \right) - \Omega_p(z_j, \lambda) \right\} \right] \geq {s}_j , \quad \forall j } \\
			\qquad \lambda \geq 0
		\end{cases}
	\end{align}
	where the regularization term $\Omega_p(z_j, \lambda)$ satisfies
	$$
	\Omega_p(z_j, \lambda) := 
	\begin{cases} 
		\frac{1}{q}(p\lambda)^{1-q} \|z_j\|_*^q & \text{if } p > 1, \\ 0 & \text{if } p=1, \, \|z_j\|_* \le \lambda, \\ \infty & \text{otherwise}
	\end{cases}
	$$	which is derived from the scaled convex conjugate of the power norm function.
	Here, the last equality~\eqref{eq: double dual norm representation} holds by using the biconjugate identity $\lambda \|x - \widehat{\mathcal{X}}_j\|^p = \sup_{z_j} \{ z_j^\top ( x - \widehat{\mathcal{X}}_j) - \Omega_p(z_j, \lambda)\}$, which is a direct application of the Fenchel–Moreau Theorem, e.g., see \cite{beck2017first}.  
	Hence, we rewrite~\eqref{eq: double dual norm representation} further as follows: 
	\begingroup
	\allowdisplaybreaks
	\begin{align}
		&\begin{cases} 
			\displaystyle\sup_{\lambda, {s}_j}  \, -\lambda\varepsilon_n^p + \frac{1}{N_n}  \sum_{j=1}^{N_n}  {s}_j \\
			\text{s.t. }  \displaystyle \inf_{ x \in \mathfrak{X}_n }  \;
			\displaystyle \sup_{ z_j} \left[ U(\Phi_n(u, x))  +  z_j^\top ( x - \widehat{\mathcal{X}}_n^{(j)})- \Omega_p(z_j, \lambda) \right] \geq  {s}_j , \quad \forall j  \\
			\qquad \lambda \geq 0   \notag
		\end{cases}\\
		&	\geq \begin{cases} \label{ineq:turn into min_sup}
			\displaystyle\sup_{\lambda, {s}_j} \, -\lambda\varepsilon_n^p + \frac{1}{N_n}  \sum_{j=1}^{N_n}  {s}_j \\
			\text{s.t. } { { 
					\displaystyle \sup_{ z_j} \; \inf_{ x \in  \mathfrak{X}_n }   \left[ U(\Phi_n(u, x))  +  z_j^\top ( x - \widehat{\mathcal{X}}_n^{(j)}) - \Omega_p(z_j, \lambda)  \right] \geq  {s}_j , \quad \forall j} } \\
			\qquad \lambda \geq 0
		\end{cases} \\
		& = \begin{cases}
			\displaystyle \sup_{\lambda, {s}_j,z_j} \, -\lambda\varepsilon_n^p + \frac{1}{N_n}   \sum_{j=1}^{N_n}  {s}_j \\
			\text{s.t. }   \
			\displaystyle \inf_{x \in  \mathfrak{X}_n } \left[ U(\Phi_n(u, x))  + z_j^\top ( x  - \widehat{\mathcal{X}}_n^{(j)})   \right]- \Omega_p(z_j, \lambda) \geq  {s}_j, \quad \forall j \\
			\qquad \lambda \geq 0 \\
		\end{cases}
	\end{align}
	\endgroup
	where Inequality~\eqref{ineq:turn into min_sup} follows from the general minimax inequality ($\inf_{\mathcal{X}} \sup_{z_j} [\cdot] \ge \sup_{z_j} \inf_{\mathcal{X}} [\cdot]$) applied to the constraint's core term.
	Note that strict equality in~\eqref{ineq:turn into min_sup} is not guaranteed because the objective is concave in $x$ (by Assumption \ref{ass:U_phi}), failing the convexity requirement for Sion's theorem in the minimization variable. Thus, enforcing the stronger condition (the last system above) yields a valid lower bound.
	
	Since $\mathfrak{X}_n$ is compact and the objective is continuous, the Weierstrass Extreme Value Theorem indicates that the infimum is attained. Substituting this back into the maximization problem over $u$ yields Problem~\eqref{problem:DRO-ELG dual}.  Since the feasible set is closed and the domain $\mathcal{U}_{\rm v}(n;\eta)$ is compact, the existence of an optimal solution is guaranteed.
	
	To complete the proof, it remains to show that problem above is a convex program. Note that the objective function is linear in $\lambda, s_j$. The set $ \mathcal{U}_{\rm v}(n;\eta)$ is convex. The constraint
	$$
	G_j(u, z_j, \lambda) := \min_{x \in \mathfrak{X}_n} \left[ U(\Phi_n(u, x)) + z_j^\top ( x - \widehat{\mathcal{X}}_n^{(j)}) \right] - \Omega_p(z_j, \lambda)  \geq s_j
	$$
	defines a convex feasible set. The first term is a pointwise minimum of functions concave in $(u, z_j)$, which is concave.  (since $U(\Phi_n(\cdot, x))$ is concave by Assumption \ref{ass:U_phi} and the term linear in $z_j$ is concave, and regularization term $ \Omega_p(z_j, \lambda) \propto  \tfrac{ \|z\|_*^q }{ \lambda^{q-1}} $ is the perspective function of the convex power function $g(z) = \|z\|_*^q$. Since perspective functions preserve convexity, $-\Omega_p(z_j, \lambda) $ is concave). The pointwise minimum of concave functions is concave. Thus, $G_j$ is a sum of concave functions, and the superlevel set condition~$G_j(u, z_j) \ge s_j$ defines a convex set.  For $p=1$, the constraint reduces to $\|z_j\|_* \le \lambda$, which is also convex.
	Moreover, the constraint~$\lambda \geq 0$ is convex. Since the intersection of two convex sets preserves convexity, the overall problem is a convex program.
\end{proof}

\begin{remark}[On the Duality and the Resulting Gap] \rm
	$(i)$  The derivation of the tractable formulation in Theorem~\ref{theorem:convex reduction} involves a critical minimax interchange (see Inequality~\eqref{ineq:turn into min_sup}).  Because the utility $U(\Phi_n(u, x))$ is concave in the minimization variable~$x$ (under Assumption~\ref{ass:U_phi}), the convexity conditions for Sion's Minimax Theorem are not met, precluding an exact equality. Consequently, the use of the general minimax inequality is necessary, and the formulation in Theorem~\ref{theorem:convex reduction} furnishes a rigorous \emph{lower bound} on the true optimal value. The magnitude of this duality gap is intrinsically linked to the degree of smoothness of stage reward $x\mapsto U(\Phi_n(u,x))$ over the support set~$\mathfrak{X}_n$, uniformly over $u$. 
	$(ii)$ The convex relaxation in Theorem~\ref{theorem:convex reduction} is \emph{exact} whenever the stage reward $x \mapsto r_n(u; x)$ is affine on $\mathfrak{X}_n$; see Proposition~\ref{proposition: Minimax Duality Gap Bound}. 
\end{remark}

The next corollary indicates that the semi-infinite convex relaxation above can be further expressed in a finite convex formulation.

\begin{corollary}[Reduction of Semi-Infinite Constraint] \label{corollary: Reduction of Semi-Infinite Constraint}
	Fix $p \in [1, \infty)$, $n \in \mathcal{N}$, and let $\varepsilon_n$ be as in Definition~\ref{def: calibrated_radius}. The convex approximation problem~\eqref{problem:DRO-ELG dual} is equivalent to the following optimization problem:
	\begin{align} 
		&\sup_{u, \lambda, s_j, z_j}\, \frac{1}{n}\, \left(- \lambda\varepsilon_n^p + \frac{1}{N_n} \sum_{j=1}^{N_n}  {s}_j \right) \notag \\ 
		&\text{s.t. }  \
		\min_{ x \in \operatorname{Ext}( \operatorname{conv} ( \mathfrak{X}_n ) ) } \left[ U(\Phi_n(u, x))  +  z_j^\top ( x - \widehat{\mathcal{X}}_n^{(j)} ) \right] -\Omega_p(z_j,\lambda) \geq  {s}_j, \quad \forall j, \label{eq:ext_constraint}\\
		& \qquad \lambda \geq 0, \notag\\ 
		& \qquad u \in  \mathcal{U}_{\rm v}(n;\eta), \notag
	\end{align}
	where $\operatorname{Ext}(\operatorname{conv} ( \mathfrak{X}_n )  )$ is the set of extreme points of the convex hull of the compact support~$\mathfrak{X}_n$.
	Moreover, if the support set $\mathfrak{X}_n$ is a convex polytope, then $\operatorname{Ext}(\operatorname{conv} ( \mathfrak{X}_n )  ) = \operatorname{Ext}(\mathfrak{X}_n)$ is a finite set of vertices, and the problem reduces to a finite-dimensional convex programming problem.
\end{corollary}
\begin{proof}
	Fix $u, z_j$, and a sample index $j$. Define an auxiliary function:
	$$
	\psi_j (x) := U(\Phi_n(u, x)) + z_j^\top ( x - \widehat{\mathcal{X}}_n^{(j)} ).
	$$
	This is a continuous concave function over the compact set $\mathfrak{X}_n$. Hence, the minimum is attained by the Weierstrass Extremum Theorem.
	
	Notably, since $\mathfrak{X}_n$ is compact, its convex hull $C:=\operatorname{conv}(\mathfrak{X}_n)$ is a compact convex set. 
	We invoke the fundamental result in convex analysis that a concave function $\psi_j$ attaining a minimum over $C$ attains that minimum at one of its extreme points, see \cite[Proposition 2.4.1]{bertsekas2009convex}. 
	Therefore, we have
	\begin{align} \label{eq: min of concave function over convex domain}
		\min_{ x \in \operatorname{conv}(\mathfrak{X}_n)} \psi_j(x) = \min_{ x \in \operatorname{Ext}(\operatorname{conv}(\mathfrak{X}_n))} \psi_j(x).
	\end{align}
	Furthermore, since $\operatorname{Ext}(\operatorname{conv}(\mathfrak{X}_n)) \subseteq \mathfrak{X}_n \subseteq \operatorname{conv}(\mathfrak{X}_n)$, the minimum over the convex hull is equivalent to the minimum over the original set $\mathfrak{X}_n$, establishing the desired~result:
	\begin{align} \label{eq: min of concave function over convex domain }
		\min_{ x \in \mathfrak{X}_n} \psi_j(x) 
		= \min_{ x \in \operatorname{Ext}(\operatorname{conv}(\mathfrak{X}_n))} \psi_j(x),
	\end{align}
 Therefore, \eqref{eq: min of concave function over convex domain } implies 
	$
	\min_{ x \in \mathfrak{X}_n} \psi_j(x) - \Omega_p(z_j, \lambda)= \min_{ x \in \operatorname{Ext}(\operatorname{conv} (\mathfrak{X}_n))} \psi_j(x)- \Omega_p(z_j, \lambda),
	$ 
	since $\Omega_p(\cdot)$ is $x$-independent.
	Consequently, the semi-infinite constraint $\min_{x} \psi_j(x) -\Omega_p(z_j, \lambda) \geq s_j$ is satisfied if and only if it holds for all $v \in \operatorname{Ext}( \operatorname{conv}(\mathfrak{X}_n))$. If $\mathfrak{X}_n$ is a polytope, $|\operatorname{Ext}(\mathfrak{X}_n)| < \infty$, ensuring a finite number of constraints.
\end{proof}

\begin{remark}[Computational Complexity]\rm
	Corollary~\ref{corollary: Reduction of Semi-Infinite Constraint} reduces the semi-infinite constraint
	indexed by $\mathfrak{X}_n$ to an equivalent constraint indexed by the extreme-point set
	$\operatorname{Ext}( \operatorname{conv}( \mathfrak{X}_n))$. In general, $\operatorname{Ext}( \operatorname{conv}( \mathfrak{X}_n))$ may be infinite (e.g., when $\mathfrak{X}_n$ is strictly convex), and the resulting formulation remains semi-infinite.
	If, however, the $\operatorname{conv}(\mathfrak{X}_n)$ is a polytope, then $\operatorname{Ext}( \operatorname{conv}( \mathfrak{X}_n))$ is a finite vertex set. In this case, the number of constraints scales as $N_n \times |\operatorname{Ext}(\mathfrak X_n)|$, which can still
	grow exponentially with the disturbance dimension $d$ (e.g., for a hypercube, $|\operatorname{Ext}(\mathfrak X_n)|=2^d$).
	When $|\operatorname{Ext}(\mathfrak X_n)|$ is large, a cutting-plane method (see Appendix~\ref{appendix: cutting_plane})
	can be used to solve the problem by iteratively adding only violated extreme-point constraints.
\end{remark}

\subsection{Theoretical Guarantees and Minimax Duality Gap Analysis} \label{subsection:duality_gap_analysis}
The tractable formulation in Theorem~\ref{theorem:convex reduction} provides a computable \emph{lower bound} on the true optimal value. This bound arises from the use of the general minimax inequality~\eqref{ineq:turn into min_sup}, necessitated by  the concavity of the utility function in the disturbance variable~$x$. Consequently, the practical effectiveness of the formulation is governed by the magnitude of the resulting minimax duality gap.

Theoretically, this gap is driven by the smoothness (non-linearity) of the utility function over the support set~$\mathfrak{X}_n$. The following proposition formalizes this by providing an explicit, computable bound based on the smoothness constant $L_n$ introduced in Assumption~\ref{ass:U_phi}.

\begin{proposition}[Minimax Duality Gap Bound] \label{proposition: Minimax Duality Gap Bound} 
	Fix $u \in  \mathcal{U}_{\rm v}(n;\eta)$ and $\lambda \geq 0$. Let $\mathfrak{D} := \sup_{x, y \in \mathfrak{X}_n} \|x - y\|$ be the diameter of the disturbance support set under the ground $\ell_r$-norm $\|\cdot\|$. For any order $p \in [1, \infty)$, the minimax duality gap for each data sample $j$, defined as
	\begin{align*}
	\Delta_j(u, \lambda) 
	&:= \inf_{ x \in \mathfrak{X}_n } \left[ U(\Phi_n(u, x))  + \lambda \| x - \widehat{\mathcal{X}}_n^{(j)}\|^p \right] \\
	&\qquad - \sup_{z_j \in \mathbb{R}^d} \; \inf_{ x \in \mathfrak{X}_n } \left[ U(\Phi_n(u, x))  + z_j^\top(x - \widehat{\mathcal{X}}_n^{(j)}) - \Omega_p(z_j, \lambda) \right], 
\end{align*}
	is bounded from above by
	$
	\Delta_j(u, \lambda) \le \frac{1}{2} L_n \mathfrak{D}^2,
	$
	where $L_n$ is the uniform smoothness bound from Assumption~\ref{ass:U_phi}(A3).
	Moreover, if the composite stage reward map $x \mapsto r_n(u,x)$ is affine on $\mathfrak{X}_n$, then the relaxation is exact, i.e., $\Delta_j =0.$
\end{proposition}
\begin{proof} 
	To prove that the minimax duality gap is bounded from above, we define the primal and dual values associated with the inner variational problem.
	For a fixed $u \in  \mathcal{U}_{\rm v}(n;\eta)$, $\lambda \geq 0$, and data sample $\widehat{\mathcal{X}}_n^{(j)}$, let $f_u(x) := U(\Phi_n(u, x))$. 
	By Assumption~\ref{ass:U_phi}(A2), $f_u$ is concave on $\mathfrak{X}_n$.
	By Assumption~\ref{ass:U_phi}(A3), $f_u$ is differentiable on $\mathfrak{X}_n$ and $\nabla f_u$ is $L_n$-Lipschitz on $\mathfrak X_n$.
	Using the Fenchel representation $\lambda \|v\|^p = \sup_z \{z^\top v - \Omega_p(z, \lambda)\}$ as derived in the proof of Theorem~\ref{theorem:convex reduction},  the primal value of the inner problem as:
	\begin{align} \label{eq: primal value of the inner problem}
		P_j 
		&:= \inf_{x \in \mathfrak{X}_n } \left[ f_u( x ) + \lambda \| x - \widehat{\mathcal{X}}_n^{(j)}\|^p \right] \notag \\
		&= \inf_{ x \in \mathfrak{X}_n } \sup_{z_j \in \mathbb{R}^d} \left[ f_u( x ) + z_j^\top( x - \widehat{\mathcal{X}}_n^{(j)}) - \Omega_p(z_j, \lambda) \right].
	\end{align}
	Define the dual value by exchanging infimum and supremum as:
	\begin{align} \label{eq: dual value of the inner problem}
		D_j := \sup_{z_j \in \mathbb{R}^d} \inf_{x \in \mathfrak{X}_n } \left[ f_u(x) + z_j^\top(x - \widehat{\mathcal{X}}_n^{(j)}) - \Omega_p(z_j, \lambda) \right].
	\end{align}
	The minimax duality gap is the difference, $\Delta_j := P_j - D_j$. The general minimax inequality states that $\inf\sup[\cdot] \ge \sup\inf[\cdot]$, which implies that $\Delta_j \ge 0$. Our remaining task is to find an upper bound for~it.
	
	To upper bound the gap, we linearize the concave function $f_u(x)$. Fix an arbitrary linearization point $y \in \mathfrak{X}_n$. Let $T_y(x) := f_u(y) + \nabla f_u(y)^\top(x-y)$, which represents the tangent hyperplane to $f_u$ at the point $y$.  Since $f_u(x)$ is concave, it lies below its tangent:
	\begin{align} \label{eq:concavity_bound}
		f_u(x) \leq T_y(x), \quad \text{ for all } x \in \mathfrak{X}_n.
	\end{align}
	The remainder $R(x, y) := T_y(x) - f_u(x) \geq 0$. 
	
	Additionally, by Assumption~\ref{ass:U_phi}(A3), gradient $\nabla f_u$ is $L_n$-Lipschitz in the sense that
	$\|\nabla f_u(x)-\nabla f_u(y)\|_* \leq L_n \|x-y\|$ for all $x, y \in \mathfrak{X}_n$.
	Applying the (norm-generalized) descent lemma, e.g., see \cite[Lemma 5.7]{beck2017first} to the convex function $-f_u$ yields the uniform bound
	\begin{equation}\label{eq:remainder_bound}
		0 \leq R(x,y) \leq \frac{1}{2} L_n \|x-y\|^2 \leq \frac{1}{2}L_n \mathfrak{D}^2,
		\qquad \text{for all } x, y \in \mathfrak{X}_n,
	\end{equation}
	where $\mathfrak{D}:=\sup_{x,y\in\mathfrak X_n}\|x-y\|$.
	Now, substituting $f_u(x) = T_y(x) - R(x, y)$ into the expressions for $P_j$ and $D_j$, we have the following bounds:
	
	\emph{Bound the Primal:} Using \eqref{eq:concavity_bound}, we substitute $T_y$ for $f_u$:
	\begin{align*}
		P_j 
		&= \inf_{ x \in \mathfrak{X}_n } \sup_{z_j \in \mathbb{R}^d} \left[ f_u( x ) + z_j^\top( x - \widehat{\mathcal{X}}_n^{(j)}) - \Omega_p(z_j, \lambda) \right]\\
		&\leq \inf_{ x \in \mathfrak{X}_n} \sup_{z_j \in \mathbb{R}^d} \left[ T_y(x) +z_j^\top( x - \widehat{\mathcal{X}}_n^{(j)}) - \Omega_p(z_j, \lambda) \right] =: P_{\rm lin}(y).
	\end{align*}
	\emph{Bound the Dual:} Since $\Omega_p(z_j, \lambda)$ is independent of $x$, the remainder term separates:
	\begin{align*}
		D_j
		&= \sup_{z_j \in \mathbb{R}^d} \inf_{x \in \mathfrak{X}_n} \left[ T_y(x) - R(x, y) + z_j^\top(x - \widehat{\mathcal{X}}_n^{(j)})- \Omega_p(z_j, \lambda) \right] \\
		&\ge \sup_{z_j \in \mathbb{R}^d} \left( \inf_{x \in \mathfrak{X}_n} \left[ T_y(x) + z_j^\top(x - \widehat{\mathcal{X}}_n^{(j)}) - \Omega_p(z_j, \lambda) \right] - \sup_{x' \in \mathfrak{X}_n} R(x', y) \right) \\
		&=\sup_{z_j \in \mathbb{R}^d}  \left( \inf_{x \in \mathfrak{X}_n} \left[ T_y(x) + z_j^\top(x - \widehat{\mathcal{X}}_n^{(j)}) \right] - \Omega_p(z_j, \lambda)\right) - \sup_{x' \in \mathfrak{X}_n} R(x', y) \\
		&:= D_{\rm lin}(y) - \sup_{x' \in \mathfrak{X}_n} R(x', y),
	\end{align*}
	where $D_{\rm lin}(y) := \sup_{z_j \in \mathbb{R}^d}    \inf_{x \in \mathfrak{X}_n} \left[ T_y(x) + z_j^\top(x - \widehat{\mathcal{X}}_n^{(j)}) - \Omega_p(z_j, \lambda)\right] $.
	
	\emph{Bound the Gap:}
	The terms $P_{\rm lin}(y)$ and $D_{\rm lin}(y)$ represent the primal and dual of a minimax problem involving a function $\mathfrak{L}(x, z_j) = T_y(x) + z_j^\top(x - \widehat{\mathcal{X}}_n^{(j)}) - \Omega_p(z_j, \lambda)$. Note that $\mathfrak{L}(x, z_j)$ is affine in $x$ and {concave} in $z_j$.
Since the extrema of an affine function over a compact set $\mathfrak{X}_n$ coincide with those over its convex hull $\operatorname{conv}(\mathfrak{X}_n)$, we may consider the problem on the domain $\operatorname{conv}(\mathfrak{X}_n)$ without changing the optimal values. As $\operatorname{conv}(\mathfrak{X}_n)$ is compact and convex, Sion's Minimax Theorem applies to this extension, implying $P_{\rm lin}(y) = D_{\rm lin}(y)$.
	
	Thus, we have
	\begin{align} \label{ineq: delta_bound}
			\Delta_j = P_j - D_j \leq P_{\rm lin}(y) - \left( D_{\rm lin}(y) - \sup_{x' \in \mathfrak{X}_n} R(x', y) \right) = \sup_{x' \in \mathfrak{X}_n} R(x', y).
	\end{align}
	Using the uniform bound from \eqref{eq:remainder_bound}, we obtain $\Delta_j \le \frac{1}{2} L_n \mathfrak{D}^2$.
	
	To complete the proof, consider the case where the stage reward $f_u(x)$ is affine in $x$, i.e., $f_u(x) := a(u)^\top x + b(u)$. Then, for any linearization point $y$, the remainder term is exact. That is,
	\begin{align*}
		R(x,y) 
		= T_y(x) - f_u(x)  
		&= f_u(y) + a(u)^\top (x-y) - f_u(x) \\
&= \left[ a(u)^\top y + b(u) \right]  +a(u)^\top (x-y) - \left[ a(u)^\top x + b(u) \right] \\
		& = a(u)^\top y  +  a(u)^\top (x-y) - a(u)^\top x  =0.
	\end{align*}
	Substituting this into the bound~\eqref{ineq: delta_bound}, we conclude $\Delta_j \leq \sup_{x \in \mathfrak{X}_n} R(x, y) = 0.$
\end{proof}

\begin{remark}[On the Generality of the Duality Gap Bound] \rm
	Note that the bound in Proposition~\ref{proposition: Minimax Duality Gap Bound} is independent of the Wasserstein order $p$ and the radius $\varepsilon_n$;  it depends only on the smoothness constant $L_n$ and the diameter~$\mathfrak{D}$ of~$\mathfrak{X}_n$.
	Moreover, the bound is a worst-case \emph{uniform} estimate over $\lambda \geq 0$.
	We also note that when $\lambda=0$, feasibility forces $z_j=0$ (for $p=1$ via $\|z_j\|_*\le \lambda$, and for $p>1$ via the extended-value convention $\Omega_p(z,0)=+\infty$ for $z \neq 0$), hence $P_j=D_j$ and~$\Delta_j(u,0)=0$.
\end{remark}

We now specialize the general duality bound to the widely used affine-logarithmic structure considered in Example~\ref{example: log-optimal portfolio control}.

 \begin{corollary}[Explicit Duality Gap Bound for Affine-Logarithmic Structures] \label{corollary: log_optimal_gap}
 	Fix~$n\in\mathcal{N}$, $u \in \mathcal U_{\rm v}(n;\eta)$, and $\lambda\ge 0$.
 	Consider the setting of Example~\ref{example: log-optimal portfolio control} with
 	$\Phi_n(u,x) = u^\top x + c_n(u)$ and $U(\cdot)=\log(\cdot)$.
 	Define the realized margin
 	$
 	\underline{\Phi}_n(u) :=\min_{x\in\mathfrak X_n}\Phi_n(u,x).
 	$ 
 	Then, for each $j=1, \dots, N_n$,
 	\[
 	\Delta_j(u, \lambda) \leq \frac{\|u\|_*^2}{2 \underline{\Phi}_n(u)^2}\,\mathfrak{D}^2 
 		 \leq \frac{\|u\|_*^2}{2\,\eta^2}\,\mathfrak{D}^2.
 	\]
 \end{corollary}
 \begin{proof}
	Since $u\in\mathcal U_{\rm v}(n;\eta)$, by definition we have $\Phi_n(u,x)\ge \eta$ for all $x\in\mathfrak X_n$. Taking the minimum over $x\in\mathfrak X_n$ yields $\underline{\Phi}_n(u)=\min_{x\in\mathfrak X_n}\Phi_n(u,x)\ge \eta$.
	 Moreover, as derived in Example~\ref{example: log-optimal portfolio control}, $\nabla_x r_n(u,\cdot)$ is $L_n(u)$-Lipschitz with $L_n(u) := \tfrac{\|u\|_*^2}{\underline\Phi_n(u)^2}$. 
 	Hence, 
 	\[
 	L_n(u) =\frac{\|u\|_*^2}{\underline{\Phi}_n(u)^2}
 		\leq \frac{\|u\|_*^2}{\eta^2},
 	\]
 	Substituting into Proposition~\ref{proposition: Minimax Duality Gap Bound} yields the~claim.
 \end{proof}

\begin{remark}[Magnitude and Interpretation of the Gap Bound]\rm
	Corollary~\ref{corollary: log_optimal_gap} provides a two-level certificate on the minimax duality gap.
	The first bound, in terms of the realized margin $\underline{\Phi}_n(u)=\min_{x\in\mathfrak X_n}\Phi_n(u,x)$, is an \emph{a posteriori} estimate that can be evaluated for a given feasible control $u$ and is typically much sharper.
	The second bound replaces $\underline{\Phi}_n(u)$ by the design viability margin $\eta$ and serves as a uniform \emph{a priori} envelope that holds for all $u \in \mathcal{U}_{\rm v}(n;\eta)$.
	In both cases, the bound scales quadratically in $\|u\|_*$ and $\mathfrak{D}$, and it increases as the corresponding margin ($\underline{\Phi}_n(u)$ or $\eta$) decreases toward zero.
\end{remark}

\begin{lemma}[Probabilistic Performance Guarantee] \label{lemma: statistical_bound}
	Fix $\beta\in(0,1)$ and a finite candidate set $\mathcal N$.
	For each $n\in\mathcal N$, let $\varepsilon_n$ be calibrated according to Definition~\ref{def: calibrated_radius},  i.e.,~$
	\mathbb{P} \left(\mathbb F_{\mathrm{true},n} \in \mathcal{B}_{\varepsilon_n}^{ (p) }( \widehat{\mathbb F}_n) \right) \geq 1-\frac{\beta}{|\mathcal N|}.
$
	Let $J_{\rm cvx}^*(n)$ denote the optimal value of the tractable program in Theorem~\ref{theorem:convex reduction}
	constructed from $\widehat{\mathbb F}_n$ and radius $\varepsilon_n$.
	Then, with probability at least $1-\beta$ (with respect to the sampling that generates $\{\widehat{\mathcal X}_n^{(j)}\}$),
	the following holds simultaneously for all $n\in\mathcal N$:
	\[
	J_{\rm cvx}^*(n) \leq \max_{u\in \mathcal U_{\rm v}(n; \eta) } \frac1n\,\mathbb E^{\mathbb F_{\mathrm{true},n}}\!\left[U(\Phi_n(u,\mathcal X))\right].
	\]
\end{lemma}
 \begin{proof}
	Fix $\beta\in(0,1)$ and a finite candidate set $\mathcal N$.
For each $n\in\mathcal N$, let $\varepsilon_n$ satisfy \eqref{eq:coverage_union}. Then, Inequality \eqref{eq:coverage_union} and a union bound yields
 	\[
 	\mathbb P\!\left(\forall n\in\mathcal N:\ \mathbb F_{\mathrm{true},n}\in \mathcal{B}_{\varepsilon_n}^{ (p) } (\widehat{\mathbb F}_n)\right) \geq 1-\beta.
 	\]
 	On this event, for each fixed $n$, Theorem~\ref{theorem:convex reduction} yields
 	\begin{align} \label{ineq: Jcvx}
 	J_{\rm cvx}^*(n)\ \le\ \max_{u\in \mathcal U_{\rm v}(n; \eta) }\ \inf_{\mathbb F\in\mathcal{B}_{\varepsilon_n}^{(p)}(\widehat{\mathbb F}_n)}
 	\frac1n\,\mathbb {E}^{\mathbb F}\!\left[U(\Phi_n(u,\mathcal X))\right].
 	\end{align}
 	Since $\mathbb F_{\mathrm{true},n}$ belongs to the ambiguity set on the same event, we have for every fixed~$u \in \mathcal U_{\rm v}(n; \eta) $,
 	\[
 	\inf_{\mathbb{ F}\in \mathcal{B}_{\varepsilon_n}^{ (p) } (\widehat{\mathbb F}_n)}
 	\mathbb E^{\mathbb{ F}}\!\left[U(\Phi_n(u,\mathcal X))\right]
 	\le
 	\mathbb E^{\mathbb F_{\mathrm{true}, n}} \left[ U(\Phi_n(u,\mathcal X))\right].
 	\]
 	Taking $\max_{u\in \mathcal U_{\rm v}(n; \eta) } \frac1n(\cdot)$ on both sides yields
 	\[
 	\max_{u \in \mathcal{U}_{\rm v}(n; \eta) } \inf_{\mathbb{F} \in \mathcal{B}_{\varepsilon_n}^{ (p) } (\widehat{\mathbb F}_n)}
 	\frac1n\,\mathbb E^{\mathbb F}\!\left[U(\Phi_n(u,\mathcal X))\right]
 	\le
 	\max_{u \in \mathcal{U}_{\rm v}(n; \eta) } \frac1n\,\mathbb E^{\mathbb{F}_{\mathrm{true}, n}} \left[ U(\Phi_n(u,\mathcal X))\right].
 	\]
 	Combining with~\eqref{ineq: Jcvx} completes the proof.
 \end{proof}

 \subsection{Long-Run Performance Guarantees}
 The preceding results establish guarantees for a single sampling period. In closed-loop operation, however, the robust policy $u^*$ is applied repeatedly over an infinite horizon. We now establish the link between the tractable relaxation $J_{\rm cvx}^*(n)$ and the long-run performance of the system. We first present a general result for any concave utility satisfying the standing assumptions. This ensures that the optimal value of the relaxation provides a deterministic floor for the \emph{long-run average utility rate}.

 \begin{theorem}[Long-Run Average Utility Guarantee] \label{theorem:long_run_utility}
 	Fix $n \in \mathcal{N}$. Suppose that the sequence of $n$-period aggregated disturbances $\{\mathcal{X}_{k,n}\}_{k \geq 0}$ is strictly stationary and ergodic under the true distribution $\mathbb{F}_{\mathrm{true},n}$. Let $\varepsilon_n$ be calibrated according to Definition~\ref{def: calibrated_radius}, and define the coverage event $\mathcal{E}_n := \{ \mathbb{F}_{\mathrm{true}, n} \in \mathcal{B}_{\varepsilon_n}^{(p)} (\widehat{\mathbb{F}}_n)\}$. 
 	Let~$u^*$ be an optimal control for the tractable relaxation~\eqref{problem:DRO-ELG dual} with optimal value~$J_{\rm cvx}^*(n)$. Then, conditional on the event $\mathcal{E}_n$, the long-run average utility rate satisfies
 	\begin{align} \label{eq: utility_guarantee}
 		\lim_{K \to \infty} \frac{1}{Kn} \sum_{k=0}^{K-1} U(\Phi_n(u^*, \mathcal{X}_{k,n})) \geq J_{\mathrm{cvx}}^*(n) \qquad \mathbb{F}_{\mathrm{true}, n}\text{-a.s.}
 	\end{align}
 \end{theorem}
 
 \begin{proof}
 	Define the per-period utility realization 
 	$$
 	Y_k := \frac{1}{n} U(\Phi_n(u^*, \mathcal{X}_{k,n})), \qquad k = 0,1,2,\ldots
 	$$ 
 	Since $\{\mathcal{X}_{k,n}\}_{k\geq 0}$ is strictly stationary and ergodic, and the map $x \mapsto \frac{1}{n} U(\Phi_n(u^*, x))$ is measurable, the sequence $\{Y_k\}$ is also strictly stationary and ergodic. 
 Moreover, since~$u^* \in \mathcal{U}_{\rm v}(n; \eta)$, we have $\Phi_n(u^*,x)\ge \eta$ for all $x \in\mathfrak{X}_n$. Because~$\mathfrak{X}_n$ is compact and~$x \mapsto \Phi_n(u^*,x)$ is continuous, $\Phi_n(u^*,x)$ is bounded above on $\mathfrak{X}_n$. Since~$U(\cdot)$ is continuous on $[\eta,\infty)$, it follows that $Y_k$ is bounded and hence integrable, i.e.,~$\mathbb{E}^{\mathbb{F}_{\mathrm{true},n}}[|Y_0|]<\infty$.
 	Applying the Birkhoff Ergodic Theorem yields
 	\begin{align} \label{ineq: Birkhoff Ergodic ineq}
 	\lim_{K \to \infty} \frac{1}{K} \sum_{k=0}^{K-1} Y_k
 	= \mathbb{E}^{\mathbb{F}_{\mathrm{true},n}}[Y_0]
 	= \frac{1}{n}\mathbb{E}^{\mathbb{F}_{\mathrm{true},n}}\!\left[U(\Phi_n(u^*,\mathcal{X}))\right]
 	\qquad \mathbb{F}_{\mathrm{true},n}\text{-a.s.},
 	\end{align}
 	where $\mathcal{X}$ denotes a generic aggregate disturbance distributed according to the common distribution $\mathbb{F}_{\mathrm{true},n}$ of the stationary sequence.
 	Next, recall that the tractable relaxation in Theorem~\ref{theorem:convex reduction} is derived via the minimax inequality applied at the constraint level; i.e., see \eqref{ineq:turn into min_sup}.
 Consequently, for every fixed admissible control $u$, the relaxation value provides a lower bound on the corresponding worst-case expected utility. In particular, for the relaxation optimizer $u^*$, we have
 	\begin{align} \label{ineq: Jcvx ineq}
 		 	J_{\mathrm{cvx}}^*(n) \leq \inf_{\mathbb{F} \in \mathcal{B}_{\varepsilon_n}^{(p)}(\widehat{\mathbb{F}}_n)} \frac{1}{n} \mathbb{E}^{\mathbb{F}} \left[ U(\Phi_n(u^*, \mathcal{X})) \right].
 	\end{align}
 	
 	Conditional on the event $\mathcal{E}_n$, the true distribution $\mathbb{F}_{\mathrm{true},n} \in \mathcal{B}_{\varepsilon_n}^{(p)}(\widehat{\mathbb{F}}_n)$. Thus, the infimum over the ambiguity set is a lower bound for the expectation under the true distribution; i.e.,
 	$$
 	\inf_{\mathbb{F} \in \mathcal{B}_{\varepsilon_n}^{(p)}(\widehat{\mathbb{F}}_n)} \frac{1}{n} \mathbb{E}^{\mathbb{F}} \left[ U(\Phi_n(u^*, \mathcal{X})) \right] \leq \frac{1}{n} \mathbb{E}^{\mathbb{F}_{\mathrm{true},n}}\left[ U(\Phi_n(u^*, \mathcal{X})) \right].
 	$$
 	Combining with \eqref{ineq: Birkhoff Ergodic ineq} and \eqref{ineq: Jcvx ineq}, the desired result follows.
 \end{proof}

 For the specific case of logarithmic utility, the long-run average utility rate coincides with the asymptotic capital growth rate. The following corollary formalizes this connection.

 \begin{corollary}[Long-Run Growth Rate Guarantee] \label{corollary: growth_guarantee}
 	Consider the setting of Theorem~\ref{theorem:long_run_utility}. If the utility function is logarithmic, i.e., $U(x) = \log(x)$, then conditional on the event $\mathcal{E}_n$, the long-run realized growth rate satisfies
 	$$
 	  \lim_{K \to \infty} \frac{1}{Kn} \log V_K \geq J_{\mathrm{cvx}}^*(n) \qquad \mathbb{F}_{\mathrm{true}, n}\text{-a.s.}
 	$$
 \end{corollary}

 \begin{proof}
 	Under the multiplicative dynamics $V_{k+1} = V_k \Phi_n(u^*, \mathcal{X}_{k,n})$, the state at step $K$ is given by $V_K = V_0 \prod_{k=0}^{K-1} \Phi_n(u^*, \mathcal{X}_{k,n})$.
 	Taking logarithms and dividing by the total time $Kn$, we obtain
 	$$
 	\frac{1}{Kn} \log V_K = \frac{1}{Kn} \log V_0 + \frac{1}{Kn} \sum_{k=0}^{K-1} \log \Phi_n(u^*, \mathcal{X}_{k,n}).
 	$$
 	As $K \to \infty$, the first term $\frac{1}{Kn} \log V_0$ vanishes asymptotically, and the second term corresponds precisely to the long-run average utility rate defined in \eqref{eq: utility_guarantee} with $U=\log$. The result then follows immediately from Theorem~\ref{theorem:long_run_utility}.
 \end{proof}

\subsection{Solving the Joint Optimization Problem} \rm 
\label{subsection: Solving the Joint Optimization Problem}
The overall problem~\eqref{problem: general DRO problem} requires a joint maximization over the integer-valued sampling period $n$ and the continuous control vector $u$, making it a \emph{mixed-integer program}. 
Our solution strategy leverages the finite and low-cardinality nature of the candidate set $\mathcal{N}$. We solve the problem via explicit enumeration over $n$: for each candidate $n$, we compute the corresponding optimal control $u^*(n)$ and then select the pair $(u^*, n^*)$ that attains the largest certified objective value.

The \emph{inner problem} for each fixed $n$ is the convex program formulated in Theorem~\ref{theorem:convex reduction} and Corollary~\ref{corollary: Reduction of Semi-Infinite Constraint}.  Since the number of constraints in this program (indexed by the extreme points of $\mathfrak{X}_n$) grows exponentially with the disturbance dimension $d$, enumerating them directly is often computationally infeasible. Therefore, we solve this inner problem efficiently using a cutting-plane algorithm, the technical details of which are presented in Appendix~\ref{appendix: cutting_plane}.

\section{Illustrative Examples} \label{section: Illustrative Examples}  
This section instantiates the sampled-data robust control framework, developed in Sections~\ref{section: Problem Formulation} and \ref{section: Main Results}, on the  \emph{Log-Optimal Portfolio Control problem}. This problem serves as a canonical example of a multiplicative stochastic system where state-dependent friction (transaction costs) plays a governing role. Henceforth, for illustration, we select the Wasserstein order $p=1$ and the ground norm $\|\cdot\| =\|\cdot\|_1$.

 In the context of the general dynamics \eqref{eq:mult_dynamics_general}, the system state $V_k$ represents the portfolio wealth, the control input $u_k$ corresponds to the vector of asset allocation weights, and the disturbance $\mathcal{X}_{k,n}$ models the vector of random asset returns over the sampling interval. The growth function $\Phi_n$ explicitly captures the transaction costs associated with rebalancing.

\subsection{Setup and Data}
Let $u \in \mathbb{R}_+^m$ denote the vector of portfolio weights invested in risky assets. We assume  the transaction cost associated with rebalancing is modeled as $TC(u) := \sum_{i=1}^m \kappa_i |u_i|$, where $\kappa_i \in [0, 1)$ represents the proportional friction (e.g., brokerage fees and slippage) for asset $i$. This model is consistent with real-world trading mechanics, where fees are charged on the trade value.\footnote{
	For example, trading stocks on the Taiwan Stock Exchange (TWSE) typically incurs a broker handling fee up to $0.1425\%$, plus a securities transaction tax of $0.3\%$ on sell side. Similarly, professional brokerage services such as Interactive Brokers Pro apply asset-class-specific commission schedules based on trade value or volume.
}

Consequently, the net-of-cost growth factor is given by the affine form:
\[
\Phi_n(u, x) := u^\top x + c_n(u),
\]
where $x$ is the vector of \emph{excess returns} (i.e., risky asset returns net of the risk-free benchmark), and $c_n(u)$ represents the risk-free growth net of transaction costs, defined~as
$
c_n(u) := (1 + r_{f,n}) - TC(u).
$
Here, $r_{f, n}$ denotes the \emph{$n$-period} risk-free return, obtained by converting the observed annualized
Treasury-bill yield at the rebalancing time.\footnote{
	In our experiments, we use the CBOE Interest Rate 13-Week Treasury Bill (\texttt{IRX}) as the proxy, converting the annualized yield $r^{\mathrm{ann}}_f(t_k)$ to the sampling period $n$ via $r_{f,n}(t_k) := (1 + \tfrac{ r^{\mathrm{ann}}_f(t_k)}{252} )^n - 1$.
}
This formulation decouples the baseline net risk-free drift (isolated in $c_n(u)$) from the stochastic disturbance $x$, preserving the affine structure.

\paragraph{Data Description}
We consider a portfolio comprising a risk-free asset and ten risky assets selected from the top~ten constituents of the S\&P 500 by market capitalization  at year end~2022.
The specific assets are listed in Table~\ref{table:S&P500's top 10 stocks}. 
Stock price data was obtained from Yahoo Finance for the period from January~1,~2022, to May~31, 2025.
For the risk-free rate, we use the daily annualized yield on the 13-Week U.S. Treasury Bill (\texttt{\^IRX}). Within this horizon, the average annualized risk-free rates were approximately~$2.02\%$ in 2022, $5.07\%$ in 2023, $4.97\%$ in 2024, and $4.03\%$ in early 2025 (January--May).

Notably, the dataset encompasses distinct market regimes: 2022 was characterized by a bearish market, 2023 signaled a strong recovery (bullish phase), while 2024 and early 2025 corresponded to periods of moderate expansion and volatility. This diversity provides a robust testing ground to evaluate the adaptability of the proposed DRO~approach.

\begin{table}[htbp]
	\scriptsize
	\centering
	\caption{Selected portfolio constituents (Top 10 S\&P 500 stocks by market cap in 2022 year-end}
	\label{table:S&P500's top 10 stocks}
	\begin{tabular}{ l l l  c}
		\toprule
		Rank & Company & Ticker & Percentage of  Total\\ & & &  Index Market Value (\%)\\
		\midrule
		1 & Apple Inc. & \texttt{AAPL} & 6.2\\
		
		2 & Microsoft Corporation & \texttt{MSFT} & 5.3 \\
		
		3 & Amazon.com Inc. & \texttt{AMZN} & 2.6 \\
		
		4 & Alphabet Inc. Class C & \texttt{GOOG} & 1.6 \\
		
		5 & Alphabet Inc. Class A & \texttt{GOOGL} & 1.6 \\
		
		6 & United Health Group Inc. & \texttt{UNH} & 1.5 \\
		
		7 & Johnson \& Johnson & \texttt{JNJ} & 1.4 \\
		
		8 & Exxon Mobil Corporation & \texttt{XOM} & 1.4 \\
		
		9 & Berkshire Hathaway Inc. Class B & \texttt{BRK-B} & 1.2 \\
		
		10 & JPMorgan Chase \& Co. & \texttt{JPM} & 1.3 \\
		\bottomrule
	\end{tabular}
\end{table}

\paragraph{Backtest Simulation Methodology and Control Schemes}
To validate our theory, we implement the joint optimization framework defined in Problem~\ref{problem: general DRO problem}.
We consider a finite candidate set of admissible sampling periods $\mathcal{N} := \{5, 21, 42, 63\}$, corresponding to weekly, monthly, bi-monthly, and quarterly control updates.
For any specific sampling period $n \in \mathcal{N}$ and look-back window, we construct the empirical distribution~$\widehat{\mathbb F}_n$ and calibrate the ambiguity radius $\varepsilon_n$ consistent with Definition~\ref{def: calibrated_radius}.
We then solve the tractable convex program in Theorem~\ref{theorem:convex reduction} to obtain the certified robust performance rate $J_{\rm cvx}^*(n)$. We compare two control implementation schemes:

\begin{itemize}
	\item[$(i)$] \emph{Static Sampling Scheme.}
	At the initialization of the simulation ($t=0$), we evaluate the robust performance rate for all candidates using the initial training data and select the sampling period that maximizes the worst-case lower bound:
	\[
	n^* \in \arg\max_{n\in\mathcal N} J_{\rm cvx}^*(n).
	\]
	The selected sampling period $n^*$ is then held fixed throughout the out-of-sample backtest, representing a standard sampled-data controller with a pre-optimized but static update frequency.
	
	\item[$(ii)$] \emph{Adaptive-Sampling Scheme.}
	At the beginning of each control update step $k$, we recompute certified performance rate $J_{\rm cvx}^*(n)$ for all candidate horizons $n \in \mathcal{N}$ using the current look-back window. We then dynamically select the optimal sampling period for the subsequent holding interval:
	\[
	n_k^* \in \arg\max_{n\in\mathcal N} J_{\rm cvx}^*(n).
	\]
	In this scheme, the controller adapts the sampling period $n_k^*$ online in response to evolving market volatility and friction estimates.
\end{itemize}

\medskip
The out-of-sample backtesting simulation follows a standard rolling-window procedure, starting with an initial account value of $V_0 = \$1$:

\begin{enumerate}
	\item[$(i)$] \emph{Data Construction:} 
	At the beginning of the $k$th rebalancing period, corresponding to time $t_k$, we use a fixed look-back window of length $L=252$ trading days (approximately one year). From this history $[t_k - L, t_k]$, we construct a set of $N_n = L - n + 1$ overlapping samples of $n$-period compound return vectors $\{\widehat{\mathcal{X}}_n^{(j)}\}_{j=1}^{N_n}$. These samples form the empirical distribution $\widehat{\mathbb{F}}_n$.\footnote{
		Because these $n$-period samples are overlapping, they exhibit serial dependence. Standard sufficient conditions for the finite-sample coverage requirement in Definition~\ref{def: calibrated_radius} are typically derived under i.i.d. sampling; accordingly, in our implementation the radii $\varepsilon_n$ are calibrated via block bootstrap as a practical approximation to the target coverage condition.
	}
	
	\item[$(ii)$] \emph{Robust Optimization:} 
	We solve the tractable convex program derived in Theorem~\ref{theorem:convex reduction} using the empirical distribution $\widehat{\mathbb{F}}_n$. This yields the robustly optimal control vector $u_k^*$ (portfolio weights) for the upcoming sampling period. In the adaptive-horizon strategy, this optimization is performed for all $n \in \mathcal{N}$ at each rebalancing time, and the horizon $n_k^*$ with the largest certified value is selected.

	\item[$(iii)$] \emph{Execution:} 
	The portfolio is rebalanced to the target weights $u_k^*$. Transaction costs are incurred on the rebalancing volume of the risky assets. The asset holdings are then held constant for the next $n$ trading days (zero-order hold implementation), where $n=n^*$ for the fixed-horizon strategy and $n=n_k^*$ for the adaptive-horizon strategy.

	\item[$(iv)$] \emph{Update:} 
	At the end of the sampling period, the account value $V_{k+1}$ is marked-to-market. The look-back window slides forward by $n$ days, and the process repeats from Step~1.
	
\end{enumerate}

This procedure generates a complete out-of-sample wealth trajectory. In the main frequency-selection procedure, the ambiguity radii are calibrated as $\varepsilon_n$ according to Definition~\ref{def: calibrated_radius}. Unless otherwise stated, a proportional transaction-cost rate of $10$ bps is used throughout.

\subsection{Out-Of-Sample Performance} 
Figure~\ref{fig:backtestresultsv4} compares the out-of-sample wealth trajectories across different rebalancing schemes. We evaluate both static-sampling controllers (with fixed $n^*$) and the adaptive-sampling controller (with dynamic~$n_k^*$). While the benchmarks---Buy-and-Hold and the Daily Rebalanced Equal Weight portfolio---achieve the highest nominal cumulative wealth, they suffer from significant drawdowns (approximately $20\%$; see \texttt{MDD} in Table~\ref{table:performance_metrics}).
In contrast, the proposed DRO schemes maintain competitive growth while offering superior downside-risk control. Notably, the adaptive-sampling scheme achieves the highest risk-adjusted performance (\texttt{SR} = 0.50) in this backtest simulation, validating the benefit of dynamic frequency selection.  
 Figure~\ref{fig:adaptivehorizonselections} reveals the realized sequence of sampling periods $n_k^*$ selected by the adaptive strategy under the baseline cost setting.

\begin{figure}[htbp]
	\centering
	\includegraphics[width=0.7\linewidth]{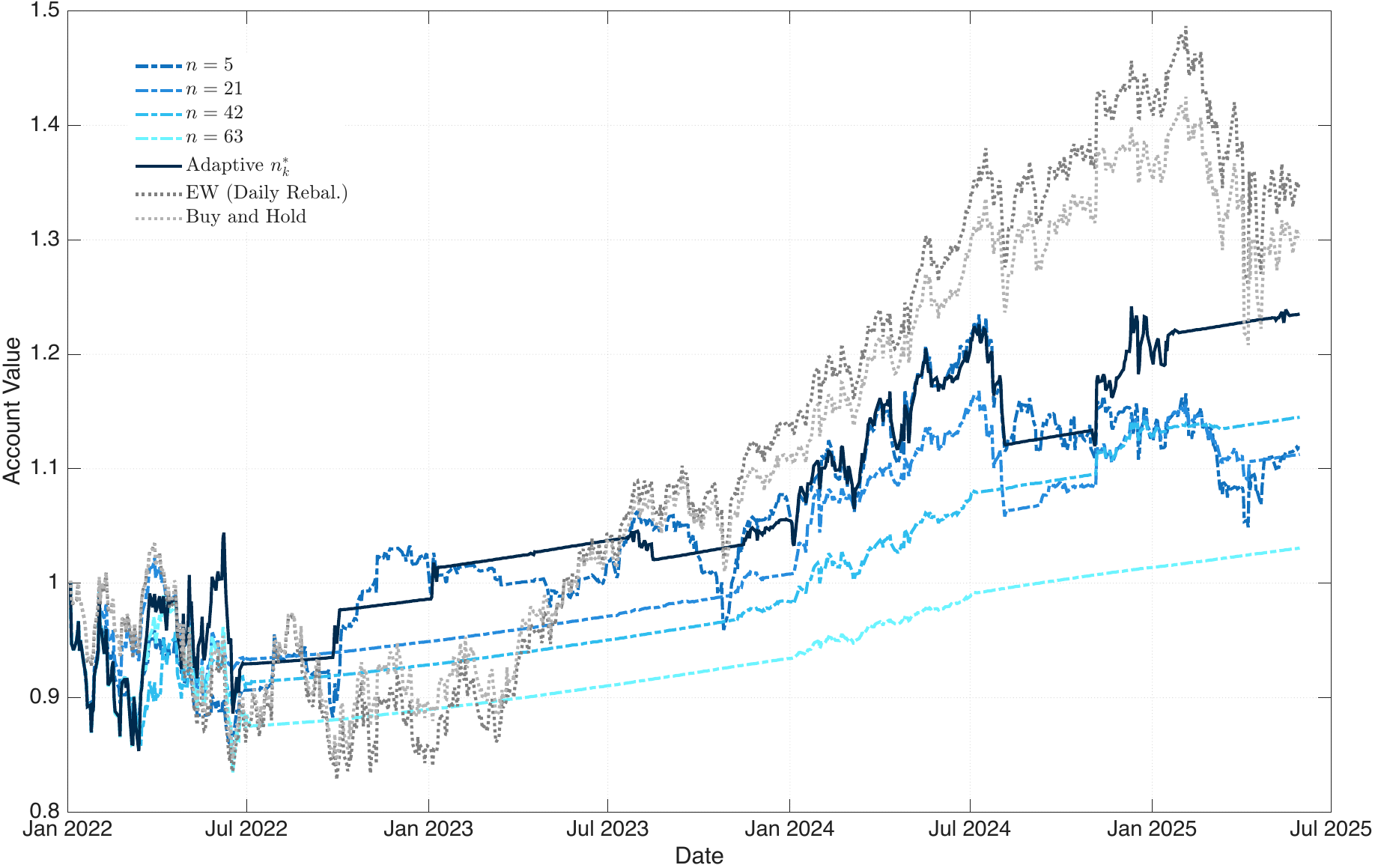}
	\caption{Out-of-sample wealth trajectories comparing static sampling with fixed $n^*$ and adaptive sampling with dynamic $n^*_k$. The DRO strategies induce a conservative allocation during high-volatility regimes, effectively limits the downside risks.}
	\label{fig:backtestresultsv4}
\end{figure}

 \begin{figure}[htbp]
 	\centering
 	\includegraphics[width=0.7\linewidth]{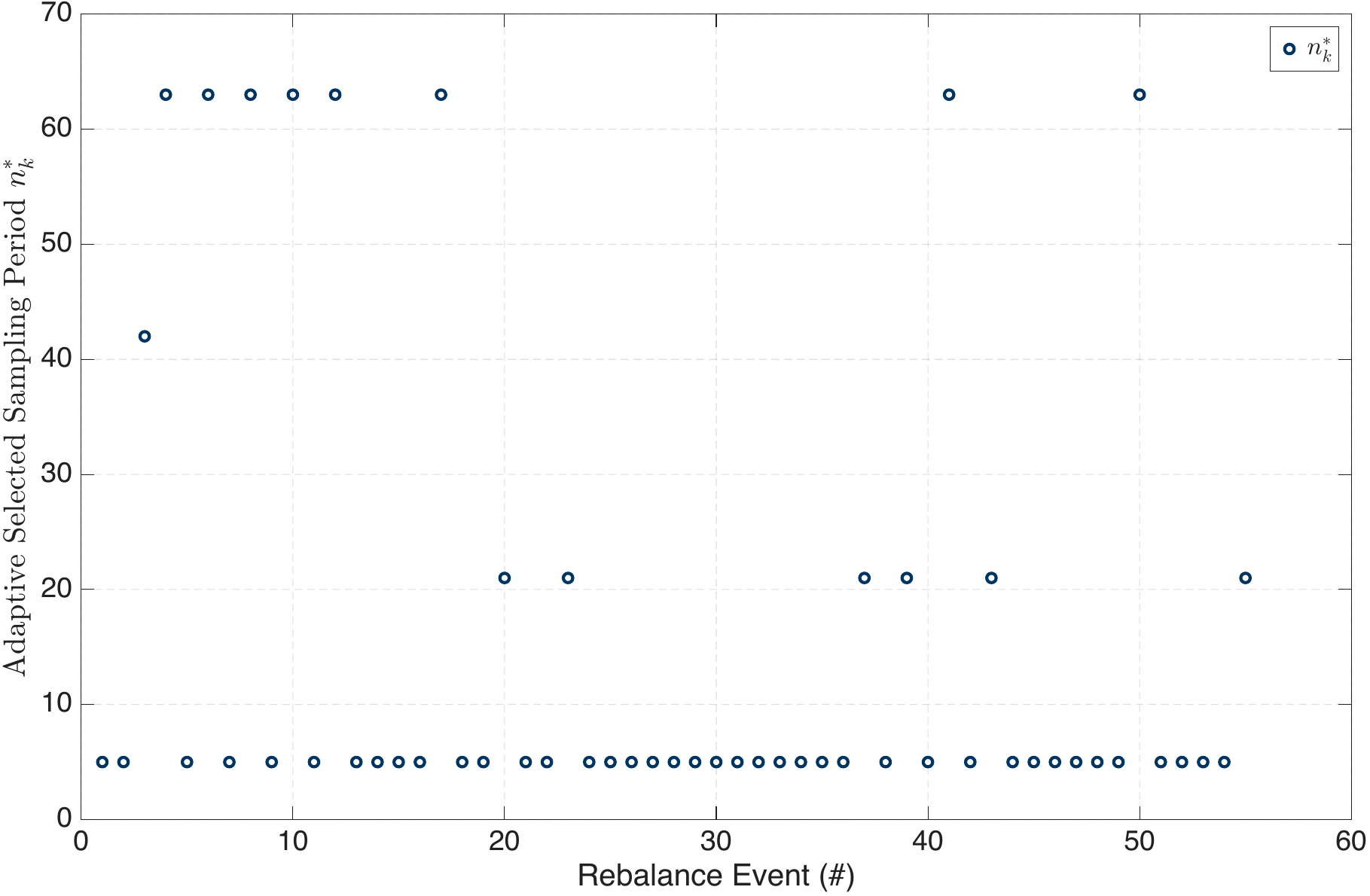}
 	\caption{Adaptive sampling scheme: realized sequence of selected sampling period $n_k^*$. The controller autonomously switches between high-frequency and low-frequency updates based on the trade-off between growth opportunities and friction (transaction costs).}
 	\label{fig:adaptivehorizonselections}
 \end{figure}

Table~\ref{table:performance_metrics} summarizes the out-of-sample trading performance across different sampling strategies and horizons. Reported metrics include final account value (\texttt{FV}),  compound annual growth rate (\texttt{CAGR}),\footnote{
	\texttt{CAGR} is computed as $\texttt{FV}^{1/T}-1$, where $T$ denotes the length of the backtest in years.
	} total return (\texttt{TR}),  maximum drawdown (\texttt{MDD}), annualized Sharpe ratio (\texttt{SR}), and annualized volatility (\texttt{Vol}), along with transaction-related statistics.

Overall, the proposed DRO schemes provide competitive risk-aware growth (see \texttt{FV}, \texttt{CAGR}, \texttt{TR}) while exhibiting improved downside-risk control, as reflected by smaller drawdowns (\texttt{MDD}) and lower volatility (\texttt{Vol}), relative to the baseline strategies (buy-and-hold and the daily rebalanced equal-weight portfolio). We observe that when estimated distributional ambiguity is high, the worst-case performance assessment becomes more conservative, prompting the controller to shift allocation toward the risk-free asset or reduce rebalancing frequency to conserve capital.

\begin{table}[htbp]
	\centering
 	\caption{Out-of-sample performance comparison of strategies.
	Shorthands: \texttt{FV} = final account value, \texttt{TR} = total return,
	\texttt{CAGR} = compound annual growth rate, \texttt{MDD} = maximum drawdown,
	\texttt{SR} = annualized Sharpe ratio, \texttt{Vol} = annualized volatility,
	\texttt{Best}/\texttt{Worst} = best/worst single-day return, \texttt{TC} = cumulative realized transaction costs paid over the backtest,
	\texttt{\#Reb} = number of rebalances. Here, a proportional transaction-cost rate of 10 bps is applied. }
	\label{table:performance_metrics}
	\scriptsize
	\begin{tabular}{lcrrrrrrrr}
		\toprule
		Strategy & $n$ 
		& \texttt{FV} & \texttt{CAGR} & \texttt{TR} & \texttt{MDD}
		& \texttt{SR} & \texttt{Vol}
 & \texttt{TC} & \texttt{\#Reb} \\
		\midrule
		Static  & 5  & 1.1183 & 3.34\% & 11.83\% &  15.07\% & 0.31	& 13.79\% & 0.0881 & 171 \\
		Static  & 21 & 1.1126 & 3.19\% & 11.26\% & 13.06\% & 0.37	& 9.90\% & 0.0112 & 41 \\
		Static  & 42 & 1.1451 & 4.06\% & 14.51\% & 14.35\% & 0.46 		& 9.77\%   & 0.0064 & 21 \\
		Static  & 63 & 1.0308 & 0.89\% & 3.08\% & 16.54\% & 0.13	& 11.68\%   & 0.0043 & 14 \\
		\midrule
		Adaptive & $n_k^*$ & 1.2352 &  6.41\% & 23.52\% & 15.14\% & 0.50 		& 14.75\% & 0.0358 & 60 \\
		EW (Daily) & 1 & 1.3504& 9.24\% & 35.04\% & 19.95\% & 0.36	& 18.62\% & 0.0092 & 854 \\
		Buy \& Hold & -- & 1.3059 & 8.16\% & 30.59\% & 18.80\% & 0.31 & 17.62\%  & 0.0000 & 0 \\
		\bottomrule
	\end{tabular}
\end{table}

We further analyze how transaction costs and distributional ambiguity jointly influence the optimal rebalancing frequency. 
Table~\ref{tab:TC_sensitvity_results_summary} shows the sensitivity of adaptive-sampling scheme. As expected, increasing transaction costs erodes performance and shifts the optimal sampling period toward longer sampling periods.  In particular, for lower TC (5 bps), the adaptive scheme favors short horizons ($n=5$ selected 77.9\% of the time) to capture transient growth opportunities. Conversely, for high TC (50 bps), the selection shifts significantly toward longer horizons ($n=63$ selected 32.4\% of the time) to mitigate friction. This behavior confirms that the proposed framework correctly balances the ``cost of control" against the ``cost of uncertainty."

\begin{table}[htbp]
	\centering
	\caption{Transaction Cost Sensitivity Summary (Adaptive Scheme). As transaction costs increase, the controller autonomously shifts toward longer sampling periods (fewer rebalances). }
	\label{tab:TC_sensitvity_results_summary}
	\scriptsize
	\begin{tabular}{l r r r r r r r}
		\toprule
		TC Rate & Final Value & Avg $n^*$ (days) & \# Rebalances 
		& $n=5$ & $n=21$ & $n=42$ & $n=63$ \\
		\midrule
		5 bps  & 1.1776 & 12.9 & 68 & 77.9\% & 10.3\% & 2.9\% & 8.8\% \\
		10 bps & 1.2352 & 15.9 & 55 & 72.7\% & 10.9\% & 1.8\% & 14.5\% \\
		25 bps & 1.0824 & 21.2 & 41 & 56.1\% & 22.0\% & 0.0\% & 22.0\% \\
		50 bps & 1.0042 & 25.6 & 34 & 55.9\% & 11.8\% & 0.0\% & 32.4\% \\
		\bottomrule
	\end{tabular}
\end{table}

\subsection{Empirical Tightness of the Minimax Relaxation}
To validate the tightness of the convex relaxation established  in Theorem~\ref{theorem:convex reduction}, we computed the theoretical upper bound on the minimax gap derived in Proposition~\ref{proposition: Minimax Duality Gap Bound} across the entire out-of-sample period. Specifically, we define the worst-case interchange error as~$\Delta_{\max} :=\max_j \Delta_j$ and verify the condition
$
\Delta_{\max} \leq  \frac{1}{2} L_n \mathfrak{D}^2 := B.
$
Since the bound~$B$ depends only on the global geometry of the ambiguity set, satisfying $\Delta_{\max} \leq B$ serves as a uniform certificate that the approximation error is bounded for every robustness constraint $j$.

\paragraph{Theoretical Validity and Utilization} 
The theoretical bound $\Delta_j \leq B$ was satisfied in 100\% of cases for sampling periods $n \geq 21$. For the short sampling periods ($n=5$), the satisfaction rate remained high at 94.7\%. The rare numerical exceptions correspond to edge cases where the theoretical bound falls below the solver's numerical tolerance.\footnote{ 
	Specifically,  we observed instances where the theoretical bound $B$ dropped to the order of $10^{-15}$ (driven by a vanishing diameter $\mathfrak{D}$ during low-volatility periods). In these cases, the solver's inherent numerical noise (approx. $10^{-9}$) technically exceeds $B$, triggering a false violation despite the gap being effectively zero.
}  Figure~\ref{fig:rollingvalidationtimeseries} illustrates the time-varying duality gap for both static- and adaptive-sampling~controls.

To assess the conservatism of the relaxation, we define  \emph{Gap Utilization} as  the ratio~$\frac{\Delta_{\max}}{B}$. Our  experiments indicate that the observed gap averages between 2.1\% and~6.4\% of the theoretical bound. This low utilization suggests that the proposed convex relaxation is empirically much tighter than the worst-case theoretical bound~implies.

\begin{figure}[htbp]
	\centering
	\includegraphics[width=0.7\linewidth]{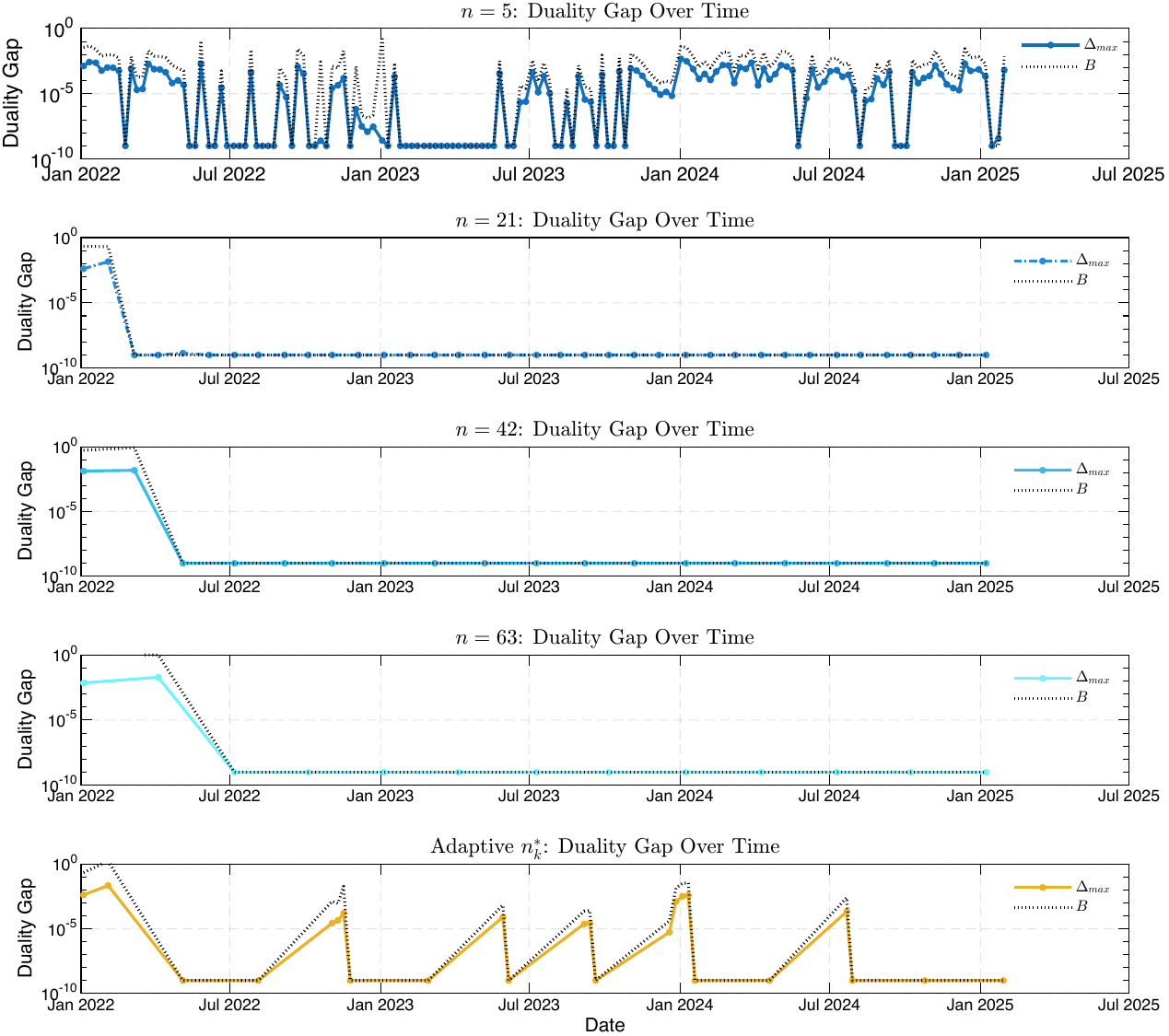}
	\caption{Time-varying duality gap for both static- and adaptive-sampling controls. The gap remains consistently small across all sampling periods $n \in \{5,21,42,63\}$ and adaptive $n_k^*$.}
	\label{fig:rollingvalidationtimeseries}
\end{figure}

\section{Conclusion} \label{section: conclusion}
	This work established a unified framework for the robust optimal control of sampled-data systems subject to multiplicative noise and distributional ambiguity. By jointly selecting the feedback policy and the control sampling period, we addressed the fundamental trade-off between discretization error, actuation costs, and model risk. A central theoretical contribution is the resolution of the ``concave-max'' geometry inherent in risk-sensitive distributionally robust control. While Sion's minimax theorem does not apply to the constraint-level interchange, we showed that the general minimax inequality yields a tractable convex relaxation that provides a rigorous lower bound on the worst-case expected utility.

	We provided four theoretical certificates to justify this framework. First, we established a \emph{probabilistic performance guarantee} (Lemma~\ref{lemma: statistical_bound}), proving that the optimal value of the tractable relaxation constitutes a valid lower confidence bound on the true expected utility with high probability.
	Second, we derived a non-asymptotic upper bound on the \emph{minimax duality gap} (Proposition~\ref{proposition: Minimax Duality Gap Bound}), demonstrating that the approximation error is uniformly controlled by the utility's smoothness and the diameter of the disturbance support.  A necessary and sufficient condition for \emph{robust viability} ensuring strict state positivity almost surely across the entire ambiguity set is also derived. 
 	Third,  and most importantly, we linked the static optimization to dynamic performance via a \emph{Long-Run Growth Rate Guarantee} (Theorem~\ref{theorem:long_run_utility}), proving that the derived robust optimal value serves as a deterministic floor for the asymptotic growth rate almost surely.

The practical efficacy of the proposed framework was demonstrated on a log-optimal portfolio control problem. The numerical results indicate that the \emph{adaptive sampling strategy}---which dynamically selects the optimal rebalancing horizon in response to evolving market conditions---significantly enhances risk-adjusted performance while maintaining system viability.
Future work could explore extending this framework to high-dimensional asset universes using factor models to mitigate the computational complexity of the semi-infinite constraints.


\bibliographystyle{siamplain}
\bibliography{refs}

\appendix

\section{Algorithmic Implementation via Cutting-Plane Method}
\label{appendix: cutting_plane}
This appendix details the cutting-plane algorithm for solving the inner-loop optimization problem in Corollary~\ref{corollary: Reduction of Semi-Infinite Constraint} for a fixed sampling period $n$. 
The resulting problem is a convex optimization problem with constraints indexed by the extreme-point
set $\operatorname{Ext}(\mathfrak{X}_n)$. When $\mathfrak{X}_n$ is a polytope, this set is finite and the number
of constraints scales as $N_n \times| \operatorname{Ext}(\mathfrak{X}_n)|$. In general, $\operatorname{Ext}(\mathfrak{X}_n)$ may be infinite, in which case the problem is semi-infinite and is solved via a separation-oracle-based cutting-plane method.

\begin{algorithm}[htbp] \footnotesize
	\caption{Cutting-Plane Algorithm for Solving the Tractable Formulation} 
	\label{alg:cutting_plane}
	\begin{algorithmic}[1]
		\State \textbf{Input:} Sampling $n \in\mathcal{N}$, samples $\{\widehat{\mathcal{X}}_n^{(j)}\}_{j=1}^{N_n}$, radius $\varepsilon_n$, and support $\mathfrak{X}_n$, and order $p \ge 1$.
		\State \textbf{Initialize:} Set iteration counter $k \leftarrow 0$. For each sample index $j \in \{1,\dots,N_n\}$, initialize the active constraint set $\mathcal V_j^{(0)}\subseteq \operatorname{Ext}(\operatorname{conv}( \mathfrak{X}_n))$ with any nonempty subset.
		
		\Loop
		\State \textbf{Solve Master Problem:} Solve the relaxed problem with the current active sets $\{\mathcal V_j^{(k)}\}$:
		\begin{align*}
			(u^{(k)},\lambda^{(k)},s^{(k)},z^{(k)}) \in   \arg& \max_{u,\lambda,s,z}\quad 
			\frac{1}{n}\left(-\lambda\varepsilon_n^p + \frac{1}{N_n}\sum_{j=1}^{N_n} s_j\right)\\
			\text{s.t.}
			& \min_{\mathcal X^v\in \mathcal V_j^{(k)}} \left[U(\Phi_n(u,\mathcal X^v)) + z_j^\top(\mathcal X^v-\widehat{\mathcal X}_n^{(j)}) - \Omega_p(z_j, \lambda)\right] \ge s_j,\quad \forall j,\\
			& \lambda\ge 0,\qquad u \in  \mathcal{U}_{\rm v}(n;\eta).
		\end{align*}
		
		\State \textbf{Separation Oracles:} For each $j\in\{1,\dots,N_n\}$, compute a most violated constraint:
		\[
		\mathcal X_j^* \in \arg\min_{\mathcal X\in\mathfrak X_n}\left[U(\Phi_n(u^{(k)},\mathcal X)) + (z_j^{(k)})^\top(\mathcal X-\widehat{\mathcal X}_n^{(j)})\right].
		\]
		Since the objective function is concave in $\mathcal{X}$ and $\mathfrak{X}_n$ is compact, the minimum is attained at an extreme point of $\operatorname{conv}(\mathfrak{X}_n)$.
		Thus, this subproblem acts as a separation oracle for the semi-infinite constraint indexed by
		$\operatorname{Ext}(\operatorname{conv}(\mathfrak{X}_n))$.
		
		\State Let $\phi_j^* := U(\Phi_n(u^{(k)},\mathcal X_j^*)) + (z_j^{(k)})^\top(\mathcal X_j^*-\widehat{\mathcal X}_n^{(j)})$.
		
		\State \textbf{Check for Convergence:} If $\phi_j^* - \Omega_p(z_j^{(k)}, \lambda^{(k)}) \ge s_j^{(k)} $ for all $j=1,\dots,N_n$, then \textbf{break}.
		
	\State \textbf{Add Cutting Planes:} For each $j$ where the condition fails:
	\State \quad Update $\mathcal V_j^{(k+1)} \leftarrow \mathcal V_j^{(k)} \cup \{\mathcal{X}_j^*\}$
	\State Otherwise set $\mathcal V_j^{(k+1)} \leftarrow \mathcal V_j^{(k)}$.
		
		\State $k \leftarrow k+1$.
		\EndLoop
		
		\State \textbf{Return} $(u^{(k)}, \lambda^{(k)})$ and the value for the given $n$.
	\end{algorithmic}
\end{algorithm}

\begin{remark}\rm
	If $\mathfrak X_n$ is a polytope, Algorithm~\ref{alg:cutting_plane} terminates in finitely many iterations,
	since only finitely many extreme-point constraints exist. For general compact $\mathfrak{X}_n$,
	the algorithm implements a standard outer-approximation method for semi-infinite convex programs,
	terminating once no violated extreme-point constraint can be found by the separation oracle.
\end{remark}

\end{document}